\documentclass{amsart}
\usepackage{amsmath, amsthm, amssymb,verbatim,fullpage,longtable,multicol,multirow}
\usepackage{graphicx}
\usepackage{hyperref}
\usepackage{soul}
\usepackage[T1]{fontenc}
\usepackage{appendix}
\usepackage{verbatim}
\usepackage[utf8]{inputenc}
\usepackage{mathtools,rotating}
\usepackage[usenames]{color}
\usepackage{tikz}
\usepackage{tkz-tab}
\usepackage{tkz-graph}
\usetikzlibrary{shapes.geometric,positioning}
\tikzstyle{vertex}=[circle, draw, inner sep=0pt, minimum size=2pt]
 \newcommand{\vertex}{\node[vertex]}
\usepackage{caption}
\newtheorem{theorem}{Theorem}[section]
\newtheorem{lemma}[theorem]{Lemma}

\newtheorem{corollary}[theorem]{Corollary}

\newtheorem{question}[theorem]{Question}

\usepackage{color}

\usepackage{newfloat}

\DeclareFloatingEnvironment[fileext=lop]{position}
\numberwithin{position}{section}

\definecolor{lightgray}{gray}{0.7}
\definecolor{midgray}{gray}{.9}
\definecolor{cof}{RGB}{219,144,71}
\definecolor{pur}{RGB}{186,146,162}
\definecolor{greeo}{RGB}{91,173,69}
\definecolor{greet}{RGB}{52,111,72}

\newcommand{\shaun}[1]{\textcolor{green}{Shaun: #1}}

\def\ZZ{{\mathbb Z}}

\def\PP{{\mathcal P}}

\def\Cycle{{\mathcal{C} }}
\def\EE{{\mathbb E}}

\title{Tipsy cop and drunken robber: a variant of the\\cop and robber game on graphs}

\author{Pamela E. Harris}
\address[P. E. Harris]{Department of Mathematics and Statistics, Williams College, Williamstown, MA United States}
\email{\textcolor{blue}{\href{mailto:peh2@williams.edu}{peh2@williams.edu}}}
\thanks{P.\,E. Harris was supported by NSF award DMS-1620202.}

\author{Erik Insko}
\address[E. Insko]{Department of Mathematics, Florida Gulf Coast University, United States}
\email{\textcolor{blue}{\href{mailto:einsko@fgcu.edu}{einsko@fgcu.edu}}}

\thanks{}

\author{Alicia Prieto-Langarica}
\address[A. Prieto-Langarica]{Department of Mathematics, Youngstown State University, United States}
\email{\textcolor{blue}{\href{mailto:aprietolangarica@ysu.edu}{aprietolangarica@ysu.edu}}}
\thanks{}

\author{Rade Stoisavljevic}
\address[R. Stoisavljevic]{Department of Mathematics, Florida Gulf Coast University, United States}
\email{\textcolor{blue}{\href{mailto:rstoisavljevic@fgcu.edu}{rstoisavljevic@fgcu.edu}}}
\thanks{}

\author{Shaun Sullivan}
\address[S. Sullivan]{Department of Mathematics, Florida Gulf Coast University, United States}
\email{\textcolor{blue}{\href{mailto:ssullivan@fgcu.edu}{ssullivan@fgcu.edu}}}
\thanks{}

\begin{document}
\maketitle

\begin{abstract}
Motivated by a biological scenario 
where a neutrophil chases a bacteria cell moving in random directions, we present a variant of the cop and robber game on graphs called the tipsy cop and drunken robber game. 
In this game, we place a tipsy cop and a drunken robber at different vertices of a finite connected graph $G$. The game consists of independent moves where the robber begins the game by moving to an adjacent vertex from where he began, this is then followed by the cop moving to an adjacent vertex from where she began. Since the robber is inebriated, he takes random walks on the graph, while the cop being tipsy means that her movements are sometimes random and sometimes intentional. 
Our main results give formulas for the probability that the robber is still free from capture after $m$ moves of this game on highly symmetric graphs, such as the complete graphs, complete bipartite graphs, and cycle graphs. We also give the expected encounter time between the cop and robber for these families of graphs. 
We end the manuscript by presenting a general method for computing such probabilities and also detail a variety of directions for future research. 
\end{abstract}

\tableofcontents

\section{Introduction} 
The original game of cops and robbers on graphs was introduced independently by
Quilliot \cite{Q86} and Nowakowski and Winkler \cite{NW83}.   In this
game, both a cop and a robber alternate turns moving from vertex to adjacent vertex on a connected graph $G$ with
the cop trying to capture the robber and the robber trying to elude
the cop.   
This formulation of the game has generated
a great deal of study \cite{1,2,3,4,AF, Gal}.

In 2014, Komarov and Winkler studied a variation of the cop and robber game in which the robber is too drunk to employ an  escape strategy, and
at each step he moves to a neighboring vertex chosen uniformly at random \cite{KW14}.
They noted that on a finite graph $G$ containing $v$ vertices even a stationary or drunk cop will win this game in an expected time at most $4v^3/27$ (plus lower-order terms)  \cite{BW,CTW} and that a greedy cop who merely moves toward the drunk robber at every step can
win in O$(v^2)$; They then noted that a smart
cop, can capture the drunk robber in expected time $v + {\rm o}(v)$ \cite{KW14}.

In this paper, we introduce another variant of the cop and robber game that we call the tipsy cop and drunken robber game on graphs. 
In this game, we place a tipsy cop and a drunken robber at different vertices of a graph $G$. The game consists of independent moves where the robber begins the game by moving to an adjacent vertex from where he began, this is then followed by the cop moving to an adjacent vertex from where she began. We assume that the robber always moves first, and hence the robber's movement takes place on odd numbered moves of the game and the cop on the even numbered moves. Since the robber is inebriated, he follows no set strategy. That is the robber simply moves uniformly and randomly on the graph. However, the cop's movement is divided into a random component and a directed component in the direction of the robber; which is our interpretation of the cop being tipsy. To define this precisely, we let $\theta\in[0,1]$ denote the proportion of the cop's moves that are random and, consequently, let $1-\theta$ be the proportion of the cop's movements that are directed towards the robber. 

One inspiration for the tipsy cop drunken robber game is the biological scenario illustrated in the YouTube video \url{ https://www.youtube.com/watch?v=Z_mXDvZQ6dU} where a neutrophil chases a bacteria cell moving in random directions.  While the bacteria moves randomly, the neutrophil's movement appears clumsy but slightly more purposeful.  Inspired by this video, we tackle the following question in this paper:
\begin{center}\emph{What is the probability that the robber is still free from capture after $m$ moves of this game?}\end{center} 
We show that this probability depends only on the starting position of the cop and robber, and we give recursive (and some closed) formulas for these probabilities for complete graphs (Section \ref{sec:complete}) and complete bipartite graphs (Section \ref{sec:completebipartite}). We also give the expected encounter time between the cop and robber for the complete graph, which is related to Bernoulli trials, and for the complete bipartite graph. In Section \ref{sec:algorithm} we propose a general method that can be used to study this game on any finite connected graph and illustrate its use on cycle graphs (Section \ref{sec:cycle}).  We end the manuscript in Section \ref{sec:open} by presenting many open problems and directions for further study. In particular, we pose a new probability question and illustrate a strategy for solving it on the family of graphs known as the friendship graphs (Section \ref{sec:friendship}).

\section{Complete Graphs}\label{sec:complete}
Throughout every section that follows we let $\theta\in[0,1]$ denote the proportion of the cop's movements that are random. We also let $P_m(G)$ denote the probability that the robber remains free after $m$ moves on the graph $G$. The following result gives the expected encounter (or capture) time in terms of the probabilities $P_m(G)$.

\begin{lemma}\label{lemma}
Let $X$ denote the random variable giving the capture time on the graph $G$. The expected encounter/capture time between the robber and the cop is given by 
\[
\mathbb{E}[X]=\sum_{m=0}^\infty P_m(G).
\]
\end{lemma}

\begin{proof}
Since the probability that the robber is captured after $m$ moves is $P_{m-1}(G)-P_m(G)$, the expected encounter time is 

\[
\mathbb{E}[X]=\sum_{m=0}^\infty m(P_{m-1}(G)-P_m(G)).
\] The result follows from the fact that the sum is telescopic and reduces to the desired form.
\end{proof}

We now begin our analysis with $K_v$, the complete graph on $v$ vertices, with some examples shown in Figure \ref{completeon4}. 
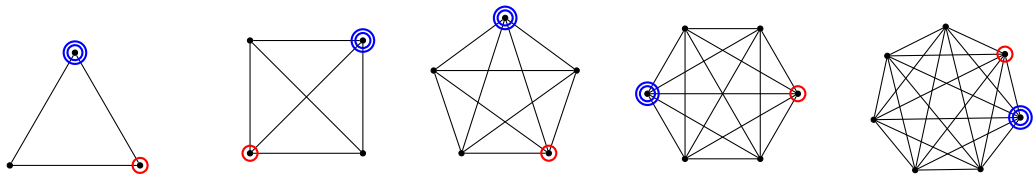
\begin{figure}[h]
\centering
\begin{turn}{-30}
\begin{tikzpicture}
\vertex[fill](a1) at (-.5000000000, .8660254040) {};
\vertex[fill](a2) at (-.5000000000, -.8660254040) {};
\vertex[fill](a3) at (1,0) {};
\draw (a1)--(a2) (a1)--(a3);
\draw (a2)--(a3);
\draw[blue,thick] (-.5,.8660) circle (.1cm);
\draw[blue,thick] (-.5,.8660) circle (.15cm);
\draw[red,thick] (1,0) circle (.1cm);
\end{tikzpicture}
\end{turn}
\hspace{5mm}
\begin{tikzpicture}
\vertex[fill](a1) at (.75,.75) {};
\vertex[fill](a2) at (-.75,.75) {};
\vertex[fill](a3) at (-.75,-.75) {};
\vertex[fill](a4) at (.75,-.75) {};
\draw (a1)--(a2) (a1)--(a3) (a1)--(a4);
\draw (a2)--(a3) (a2)--(a4);
\draw (a3)--(a4) ;
\draw[blue,thick] (a1) circle (.1cm);
\draw[blue,thick] (a1) circle (.15cm);
\draw[red,thick] (a3) circle (.1cm);
\end{tikzpicture}
\hspace{5mm}
\begin{tikzpicture}
\vertex[fill](a1) at (0,1) {};
\vertex[fill](a2) at (-.95,.30) {};
\vertex[fill](a3) at (.95,.30) {};
\vertex[fill](a4) at (-.58,-.80) {};
\vertex[fill](a5) at (.58,-.80) {};
\draw (a1)--(a2) (a1)--(a3) (a1)--(a4) (a1)--(a5);
\draw (a2)--(a3) (a2)--(a4) (a2)--(a5);
\draw (a3)--(a4) (a3)--(a5);
\draw (a4)--(a5);
\draw[blue,thick] (a1) circle (.1cm);
\draw[blue,thick] (a1) circle (.15cm);
\draw[red,thick] (a5) circle (.1cm);
\end{tikzpicture}
\hspace{5mm}
\begin{tikzpicture}
\vertex[fill](a1) at (.5000000000, .8660254040) {};
\vertex[fill](a2) at (-.5000000000, .8660254040) {};
\vertex[fill](a3) at (-1., 0.) {};
\vertex[fill](a4) at (-.5000000000, -.8660254040) {};
\vertex[fill](a5) at (.5000000000, -.8660254040) {};
\vertex[fill](a6) at (1., 0.) {};
\draw (a1)--(a2) (a1)--(a3) (a1)--(a4) (a1)--(a5) (a1)--(a6);
\draw (a2)--(a3) (a2)--(a4) (a2)--(a5) (a2)--(a6);
\draw (a3)--(a4) (a3)--(a5) (a3)--(a6) ;
\draw (a4)--(a5) (a4)--(a6);
\draw (a5)--(a6);
\draw[blue,thick] (-1,0) circle (.1cm);
\draw[blue,thick] (-1,0) circle (.15cm);
\draw[red,thick] (1,0) circle (.1cm);
\end{tikzpicture}
\hspace{5mm}
\begin{turn}{-12}
\begin{tikzpicture}
\vertex[fill](a1) at (.6234898018, .7818314825) {};
\vertex[fill](a2) at (-.2225209335, .9749279123) {};
\vertex[fill](a3) at (-.9009688678, .4338837393) {};
\vertex[fill](a4) at (-.9009688678, -.4338837393) {};
\vertex[fill](a5) at (-.2225209335, -.9749279123) {};
\vertex[fill](a6) at (.6234898018, -.7818314825) {};
\vertex[fill](a7) at (1., 0.) {};
\draw (a1)--(a2) (a1)--(a3) (a1)--(a4) (a1)--(a5) (a1)--(a6) (a1)--(a7);
\draw (a2)--(a3) (a2)--(a4) (a2)--(a5) (a2)--(a6) (a2)--(a7);
\draw (a3)--(a4) (a3)--(a5) (a3)--(a6) (a3)--(a7);
\draw (a4)--(a5) (a4)--(a6) (a4)--(a7);
\draw (a5)--(a6) (a5)--(a7);
\draw (a6)--(a7);
\draw[blue,thick] (1,0) circle (.1cm);
\draw[blue,thick] (1,0) circle (.15cm);
\draw[red,thick] (.6234,.7818) circle (.1cm);
\end{tikzpicture}
\end{turn}
\captionof{figure}{Examples of complete graphs ($K_v$ for $3\leq v\leq 7$) with cop (blue double circle) and robber (red single circle).}\label{completeon4}
\end{figure}

\begin{theorem}\label{thm:complete}If $0\leq \theta\leq 1$ denotes the proportion of the cop's movement that is random, then the probability that the robber remains free after $m$ moves on $K_v$ is given by
\begin{align}
P_m(K_v)&=\theta^{\lfloor\frac{m}{2}\rfloor}\left(\frac{v-2}{v-1}\right)^m.\label{complete}
\end{align}
\end{theorem}
\begin{proof}
Observe that on the complete graph $K_v$ the cop and robber are at all times one vertex away from each other. 
Recall that the robber moves randomly and makes the first move within the graph $K_v$. For the first move, the probability of the robber moving and not finding the cop is given by \[P_1(K_v)=\frac{v-2}{v-1}\] as the robber cannot stay put and cannot move to the location of the cop and then the game would end. If the game continues past the first move, it is the cop's turn to move. The cop moves strategically (with probability $1-\theta$) or randomly (with probability $\theta$). As she is just one step away from the robber, if she were to move strategically, she will catch the robber. If the cop moves randomly with probability $\theta$ and does not catch the robber (with probability $\frac{v-2}{v-1}$), the probability of the game continuing past the second move  is given by \[P_2(K_v)=\theta\left(\frac{v-2}{v-1}\right)^2.\] The game continues in this fashion and hence the probability that the robber remains free after $m$ moves on $K_v$ is given by the following 

\[P_m(K_v)=\theta^{\lfloor\frac{m}{2}\rfloor}\left(\frac{v-2}{v-1}\right)^m,\]
where we note that the exponent of $\theta$ is given by $\lfloor\frac{m}{2}\rfloor$ since this factor only appears when the cop moves, and she moves second in the game.
\end{proof}

Observe that when $\theta=1$ the cop is drunk and moves randomly. In this case Equation \eqref{complete} reduces to
\[P_m(K_v)=\left(\frac{v-2}{v-1}\right)^m.\]
\begin{theorem}
Let $X$ denote the random variable giving the capture time on the complete graph $K_v$.  If $\theta$ denotes the proportion of the cop's movement that is random, then the expected encounter time between the robber and cop is given by
\[\mathbb{E}[X]=\frac{(2v-3)(v-1)}{(v-1)^2-\theta(v-2)^2}.\]
\end{theorem}
\begin{proof}
By Theorem \ref{thm:complete} and Lemma \ref{lemma}, we have
\begin{align*}
\mathbb{E}[X]&=\sum_{m=0}^\infty P_m(K_v)\\
&=\sum_{m=0}^\infty \theta^{\lfloor\frac{m}{2}\rfloor}\left(\frac{v-2}{v-1}\right)^m\\
&=\sum_{m=0}^\infty \theta^{\lfloor\frac{2m}{2}\rfloor}\left(\frac{v-2}{v-1}\right)^{2m} +\sum_{m=0}^\infty \theta^{\lfloor\frac{2m+1}{2}\rfloor}\left(\frac{v-2}{v-1}\right)^{2m+1}\\    
&=\sum_{m=0}^\infty \left(\theta\left(\frac{v-2}{v-1}\right)^2\right)^{m} +\left(\frac{v-2}{v-1}\right)\sum_{m=0}^\infty \left(\theta\left(\frac{v-2}{v-1}\right)^2\right)^{m}.
\end{align*}
By the geometric sum formula we find that
\[\mathbb{E}[X]=\frac{1}{1-\theta(\frac{v-2}{v-1})^2}\left(1+\left(\frac{v-2}{v-1}\right)\right)=\frac{(2v-3)(v-1)}{(v-1)^2-\theta(v-2)^2}.\qedhere\]
\end{proof}

Below we present an analogous result in the case where a completed round consists of a movement by the robber and a movement by the cop.
\begin{theorem}
Let $X$ denote the random variable giving the capture time in number of rounds (where a round consists of one robber move and one cop move) on the complete graph $K_v$.   If $\theta$ denotes the proportion of the cop's movement that is random, then the expected encounter time between the robber and the cop is given by 
\[ \EE[X] = \frac{1}{ \frac{1}{v-1} + \frac{v-2}{v-1}\left( \left ( \frac{1}{v-1} \right )\theta +1 - \theta \right ) }=\frac{(v-1)^2}{(v-1)^2-\theta(v-2)^2}.\]
\end{theorem}

\begin{proof}
If the robber and cop are walking on a complete graph with $v$ nodes, at any given time the cop is only one step away from the robber. The probability of an encounter is given by the sum of three conditional probabilities. The first conditional probability is the probability that the robber randomly chooses the vertex occupied by the cop, which is given by $\frac{1}{v-1}$. 
The second conditional probability is the probability that the robber avoids the cop, and then the cop moves randomly and lands on the robber's vertex, which is $\left ( \frac{v-2}{v-1} \right )\theta\left (\frac{1}{v-1}\right)$. 
The final conditional probability is the probability that the robber avoids the cop, then the cop moves strategically to the vertex where the robber is located.  This probability is $ \frac{v-2}{v-1} (1- \theta)$. Therefore, at any given time the probability of the cop catching the robber can be seen as a Bernoulli trial  with probability of success equal to $P=\frac{1}{v-1} + \frac{v-2}{v-1}\left( \frac{1}{v-1}\theta +1 - \theta \right )$. Knowing this, it can be concluded that the expected time it would take the cop to catch the robber is given by the reciprocal of~$P$.
\end{proof}

\section{Complete Bipartite Graphs}\label{sec:completebipartite}
We now consider the tipsy cop and drunken robber game on complete bipartite graphs. We let $K_{v,w}$ denote the complete bipartite graph with $v, w$ denoting the size of each of the two disjoint and independent sets of vertices satisfying that every edge connects a vertex in one vertex set to the other. 
There are four distinct starting positions for the cop and robber (up to isomorphism) as depicted on $K_{3,2}$ in positions \ref{set1}-\ref{set4} below. 
\setcounter{position}{0}
\makeatletter 
\renewcommand{\theposition}{\@arabic\c@position}
\makeatother

\vspace{12pt}

\begin{minipage}{\textwidth}
\begin{minipage}[b]{0.2\textwidth}
\centering
\resizebox{.5\textwidth}{!}{
\begin{tikzpicture}
\vertex[fill](a1) at (0,1) {};
\vertex[fill](a2) at (0,2) {};
\vertex[fill](a3) at (0,3) {};
\vertex[fill](b1) at (2,1.5) {};
\vertex[fill](b2) at (2,2.5) {};
\draw (a1)--(b1) (a1)--(b2);
\draw (a2)--(b1) (a2)--(b2);
\draw (a3)--(b1) (a3)--(b2);
\draw[blue,thick] (0,3) circle (.1cm);
\draw[blue,thick] (0,3) circle (.15cm);
\draw[red,thick] (2,2.5) circle (.1cm);
\end{tikzpicture}}
\captionof{position}{}\label{set1}
\end{minipage}
\hspace{.5cm}
\begin{minipage}[b]{0.2\textwidth}
\centering
\resizebox{.5\textwidth}{!}{
\begin{tikzpicture}
\vertex[fill](a1) at (0,1) {};
\vertex[fill](a2) at (0,2) {};
\vertex[fill](a3) at (0,3) {};
\vertex[fill](b1) at (2,1.5) {};
\vertex[fill](b2) at (2,2.5) {};
\draw (a1)--(b1) (a1)--(b2);
\draw (a2)--(b1) (a2)--(b2);
\draw (a3)--(b1) (a3)--(b2);
\draw[blue,thick] (0,3) circle (.1cm);
\draw[blue,thick] (0,3) circle (.15cm);
\draw[red,thick] (0,2) circle (.1cm);
\end{tikzpicture}}
\captionof{position}{}\label{set2}
\end{minipage}
\hspace{.5cm}
\begin{minipage}[b]{0.2\textwidth}
\centering
\resizebox{.5\textwidth}{!}{
\begin{tikzpicture}
\vertex[fill](a1) at (0,1) {};
\vertex[fill](a2) at (0,2) {};
\vertex[fill](a3) at (0,3) {};
\vertex[fill](b1) at (2,1.5) {};
\vertex[fill](b2) at (2,2.5) {};
\draw (a1)--(b1) (a1)--(b2);
\draw (a2)--(b1) (a2)--(b2);
\draw (a3)--(b1) (a3)--(b2);
\draw[red,thick] (0,3) circle (.1cm);
\draw[blue,thick] (2,2.5) circle (.1cm);
\draw[blue,thick] (2,2.5) circle (.15cm);
\end{tikzpicture}}
\captionof{position}{}\label{set3}
\end{minipage}
\hspace{.5cm}
\begin{minipage}[b]{0.2\textwidth}
\centering
\resizebox{.5\textwidth}{!}{
\begin{tikzpicture}
\vertex[fill](a1) at (0,1) {};
\vertex[fill](a2) at (0,2) {};
\vertex[fill](a3) at (0,3) {};
\vertex[fill](b1) at (2,1.5) {};
\vertex[fill](b2) at (2,2.5) {};
\draw (a1)--(b1) (a1)--(b2);
\draw (a2)--(b1) (a2)--(b2);
\draw (a3)--(b1) (a3)--(b2);
\draw[blue,thick] (2,2.5) circle (.1cm);
\draw[blue,thick] (2,2.5) circle (.15cm);
\draw[red,thick] (2,1.5) circle (.1cm);
\end{tikzpicture}}
\captionof{position}{}\label{set4}
\end{minipage}
\end{minipage}\\

\noindent We let $P_m^i(K_{v,w})$ with $1 \leq i \leq 4$ denote the probability that the robber is still free after $m$ moves when the cop and robber begin the game located at position $i$ on graph $K_{v,w}$.\\

\begin{theorem}If $\theta$ denotes the proportion of the cop's movement that is random, then the probability that the robber remains free after $m$ moves on $K_{v,w}$ is given by the following closed formulas
\begin{align}
P_m^1(K_{v,w})&=\left(\frac{v-1}{v}\right)^{\lfloor\frac{m+3}{4}\rfloor}\left(\frac{w-1}{w}\right)^{\lfloor\frac{m+1}{4}\rfloor}\label{formulaP1}\\
P_m^2(K_{v,w})&=\theta^{\lfloor\frac{m}{2}\rfloor}\left(\frac{v-1}{v}\right)^{\lfloor\frac{m}{4}\rfloor}\left(\frac{w-1}{w}\right)^{\lfloor\frac{m+2}{4}\rfloor}\label{formulaP2}\\
P_m^3(K_{v,w})&=\left(\frac{v-1}{v}\right)^{\lfloor\frac{m+1}{4}\rfloor}\left(\frac{w-1}{w}\right)^{\lfloor\frac{m+3}{4}\rfloor}\label{formulaP3}\\
P_m^4(K_{v,w})&=\theta^{\lfloor\frac{m}{2}\rfloor}\left(\frac{v-1}{v}\right)^{\lfloor\frac{m+2}{4}\rfloor}\left(\frac{w-1}{w}\right)^{\lfloor\frac{m}{4}\rfloor}\label{formulaP4}.
\end{align}
\end{theorem}

 \begin{proof}

 Let $V$ and $W$ denote the two parts of the bipartite graph $K_{v,w}$ with $|V|=v$ and $|W|=w$.
 We break the proof up into two cases: 
 \begin{enumerate} 
  \item  when the cop and robber start in different parts (positions 1 and 3) , 
  and 
  \item  when the cop and robber start in the same part (positions 2 and 4).
 \end{enumerate}

We first prove case (1) when the robber and cop start in different parts. We note that in this case, the cop cannot capture the robber on her turn because they are both in the same part. 
 Thus the only time an encounter occurs is when the robber 
 moves to a vertex the cop is occupying.

 Assume the robber is located in part $W$, 
 and the cop is in part $V$ as depicted in position 1 above.    
 On turns $k \equiv 1 \mod 4$ the probability of the robber moving to a vertex the cop is not occupying is 
 $ \frac{v-1}{v}$.   Similarly, on turns $k \equiv 3 \mod 4$ the probability of the robber moving to a vertex the cop is not occupying is 
 $ \frac{w-1}{w}$.  
 As noted above, on turns $k \equiv 0 \mod 2$ the probability that the cop avoids the robber is $1$.
Thus to calculate the probability that the game continues after $m$ turns, 
one only needs to compute the number of rounds that are 
congruent to $1$ and $3 \mod 4$ as given in Equation \eqref{formulaP1}. 
Switching the starting positions of the cop and robber to those in position 3 above, 
the same argument gives the formula in Equation \eqref{formulaP3}.

Next we consider case (2).  Assume that the cop and robber begin the game in part $V$ as depicted in position 2 above. 
In this case, the robber can never be caught on his turn; so the probability of the game continuing on turns $k\equiv 1 \text{or } 3 \mod 4$ is $1$.
On turns $k \equiv 2 \mod 4$ the cop is moving from part $V$ to part $W$.  
So the probability of the cop moving to a vertex the robber is not on is $\theta \left (  \frac{w-1}{w} \right ) $.
Similarly, on turns $k \equiv 0 \mod 4$ the cop moves from part $W$ to part $v$, 
and the probability of the cop moving to a vertex the robber is not on is $\theta \left ( \frac{v-1}{v} \right ) $.
Thus to calculate the probability that the game continues after $m$ turns, 
one only needs to compute the number of rounds that are 
congruent to $0$ and $2 \mod 4$ as given in Equation~\eqref{formulaP2}. Switching the starting positions of the cop and robber to position 4 depicted above, the same argument gives the formula in 
Equation \eqref{formulaP4}.
\end{proof}

\begin{corollary}\label{cor: move together}
Let $X$ denote the random variable giving the capture time on $K_{v,w}$ and let $\EE^i[X]$ denote the expected encounter time between the robber and cop beginning the game in position $i$ on $K_{v,w}$. If $\theta$ denotes the proportion of the cop's movement that is random, then
\begin{align}
    \EE^1[X]= &  \frac{4vw -3w -v +1 }{ v+w-1}\label{exp1}\\
    \EE^2[X]=&\frac{2 ( vw - \theta v(w-1) )}{ vw - \theta^2 (v-1)(w-1) }\label{exp2}\\
    \EE^3[X]=&  \frac{4vw -3v -w +1 }{ v+w-1}\label{exp3}\\
    \EE^4[X]=&\frac{2 ( vw - \theta w(v-1) )}{ vw - \theta^2 (v-1)(w-1) }.\label{exp4}
\end{align}
\end{corollary}

\begin{proof}We have shown that on $K_{v,w}$ there are two distinct positions we must consider solely based on whether the cop and robber are on the same side of $K_{v,w}$ or on opposite sides at the start of the game. Hence we prove the formulas for the expected value when the game begins at positions \ref{set1} and \ref{set2} and note that the formulas for the expected value when the game begins at position \ref{set3} and \ref{set4} are obtained by simply interchanging $v$ and $w$ in the formulas for positions \ref{set1} and \ref{set2}, respectively.\\

\noindent
{\bf Case 1:} The game begins with the cop and robber at position \ref{set1}. Letting $x=\frac{v-1}{v}$ and $y=\frac{w-1}{w}$, note
\begin{align}\nonumber
\EE^{1}[X] & = \sum_{m=0}^\infty P_m^1(K_{v,w}) \\\nonumber             & = 1+ x+ x+ xy+xy+x^2y+x^2y + x^2y^2+ \cdots \\
           \nonumber
           & = 1 + 2x + 2xy + 2x^2y +2x^2y^2 + \cdots \\
           \nonumber
           & = -1 + 2 \sum_{i=0}^\infty (xy)^i + 2x \sum_{i=0}^\infty (xy)^i \\
           \nonumber
           & = \frac{2(1+x)}{1-xy} -1 \\
           & = \frac{1+2x+xy}{1-xy}.
\label{tosimplify1}
\end{align}

Substituting for $x=\frac{v-1}{v}$ and $y=\frac{w-1}{w}$ back in Equation \eqref{tosimplify1} yields 
\begin{align*} \EE^1[X] & = \frac{1+2 \left ( \frac{v-1}{v} \right ) + \left ( \frac{v-1}{v} \right) \left ( \frac{w-1}{w} \right) }{1-\left( \frac{v-1}{v} \right) \left (\frac{w-1}{w} \right) }  = \frac{4vw -3w -v +1 }{ v+w-1}.
            \end{align*} 
 
\noindent
{\bf Case 2:} The game begins with the cop and robber at position \ref{set2}. 
Letting $x=\frac{v-1}{v}$ and $y=\frac{w-1}{w}$, note
\begin{align}\nonumber
\EE^{2}[X] & = \sum_{m=0}^\infty P_m^2( K_{v,w}) \\\nonumber            & = 1 + 2\theta y + 2\theta^2xy + 2 \theta^3xy^2 +2\theta^4 x^2y^2 + \cdots \\
           \nonumber
           & =  2 \sum_{i=0}^\infty (\theta^2xy)^i + 2\theta y \sum_{i=0}^\infty (\theta^2 xy)^i \\
           & = \frac{1( 1+ \theta y)}{1-\theta^2 xy} . 
\label{tosimplify2}
\end{align}
Substituting $x=\frac{v-1}{v}$ and $y=\frac{w-1}{w}$ back in Equation \eqref{tosimplify2} yields
\begin{align*} \EE^2[X]  & = \frac{1\left ( 1+ \theta \left( \frac{w-1}{w} \right ) \right )}{1-\theta^2 \left( \frac{v-1}{v} \right)  \left ( \frac{w-1}{w} \right ) }  = \frac{2 ( vw - \theta v(w-1) )}{ vw - \theta^2 (v-1)(w-1) }.
                 \qedhere \end{align*} 
\end{proof}

\section{General method for identifying probabilities  and expected encounter times}\label{sec:algorithm}
In this section we present a general method that can be used to study this variant of the cop and robber game on any finite connected graph. To begin we define $R_m^i$ and $C_m^i$ to be probability that the game lasts at least $m$ moves with the robber or cop beginning the game in position $i$ on the graph $G$, respectively.

The general method will proceed as follows:

\begin{itemize}
    \item Identify the different possible starting positions for the robber and copper, and denote them $R^i$ and $C^i$ respectively.
    \item Write down recursive formulas for the probabilities $R_m^i$ and $C_m^i$ that the game lasts $m$ moves with the robber or cop beginning the game in position $i$, respectively.
    \item Use the recursive formulas for $R_m^i$ and $C_m^i$ to find generating functions $P_i(t) = \sum_{n=0}^\infty P_n^i t^n$ describing the probability that the game lasts at least $n=2m$ rounds (where we a round consists of two moves).
    \item Evaluate $P_i(t)$ at $t=1$ to find the expected number of rounds.
\end{itemize}

To illustrate this method, we apply it to cycle graphs.
\subsection{Cycle Graphs}\label{sec:cycle}
Denote the cycle graph on $v$ vertices by $\Cycle_v$.  
As before, we let $R^i_m$ and $C^i_m$ be the probabilities that the game lasts at least $m$ moves with the robber and cop beginning the game starting at position $i$. In this case $i$ represents the distance apart the cop and robber are at the beginning of the game, hence $1\leq i\leq \lfloor v/2\rfloor $. 
Our first result gives recursive formulas for even and odd length cycles.

\begin{lemma} 
If $v=2k$, then on an even-length cycle graph $\Cycle_v$, we have the following recursive formulas:
\[
R_m^i=\begin{cases}
    \frac{1}{2}C^{i+1}_{m-1}, & i=1 \\
    & \\
    \frac{1}{2}C^{i-1}_{m-1}+\frac{1}{2}C^{i+1}_{m-1}, & 1<i< v/2  \\
    & \\
     C^{v/2-1}_{m-1}, & i= v/2 
\end{cases}
\qquad\mbox{and}\qquad
C_m^i=\begin{cases}
    \frac{\theta}{2}R^{i+1}_{m-1}, & i=1 \\
    & \\
    \left(1-\frac{\theta}{2}\right)R^{i-1}_{m-1}+\frac{\theta}{2}R^{i+1}_{m-1}, & 1<i< v/2  \\
    & \\
     R^{v/2-1}_{m-1}, & i= v/2 .
\end{cases}
\]
\noindent
If $v=2k+1$, then on an odd-length cycle graph $\Cycle_v$, we have the following recursive formulas:

\[\small{
R_m^i=\begin{cases}
    \frac{1}{2}C^{i+1}_{m-1}, & i=1 \\
    & \\
    \frac{1}{2}C^{i-1}_{m-1}+\frac{1}{2}C^{i+1}_{m-1}, & 1<i<\frac{v-1}{2}  \\
    & \\
     \frac{1}{2}C^{(v-3)/2}_{m-1}+\frac{1}{2}C^{(v-1)/2}_{m-1}, & i=\frac{v-1}{2} 
\end{cases}
\qquad\mbox{and}\qquad
C_m^i=\begin{cases}
    \frac{\theta}{2}R^{i+1}_{m-1}, & i=1 \\
    & \\
    \left(1-\frac{\theta}{2}\right)R^{i-1}_{m-1}+\frac{\theta}{2}R^{i+1}_{m-1}, & 1<i<\frac{v-1}{2}  \\
    & \\
     \left(1-\frac{\theta}{2}\right)R^{(v-3)/2}_{m-1}+\frac{\theta}{2}R^{(v-1)/2}_{m-1}, & i=\frac{v-1}{2} .
\end{cases}
}
\]

\end{lemma}

\begin{proof} 
We begin by considering what can happen during the first move of an $m$-move game, when the cop and robber are a distance of more than one apart and not located on opposite sides of the cycle. More precisely, we begin our analysis with the case where $i$, the distance  between the cop and robber at the start of the game, satisfies $1<i<\lfloor \frac{v}{2} \rfloor$. In this case the following probability trees describe what may happen in a single move of the game: 
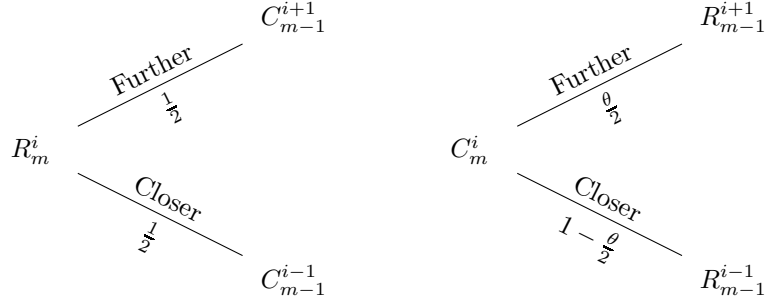
\begin{center}
\tikzstyle{level 1}=[level distance=3.5cm, sibling distance=3.5cm]
\tikzstyle{level 2}=[level distance=3.5cm, sibling distance=2cm]

\tikzstyle{bag} = [text width=4em, text centered]
\tikzstyle{end} = [circle, minimum width=3pt,fill, inner sep=0pt]

\begin{tikzpicture}[grow=right, sloped]
\node[bag] {$R^i_{m} $ }
    child {
        node[bag] {$C^{i-1}_{m-1}$}        
            edge from parent 
            node[above] {Closer}
            node[below]  {$\frac{1}{2}$}
    }
    child {
        node[bag] {$C^{i+1}_{m-1}$}        
        edge from parent         
            node[above] {Further}
            node[below]  {$\frac{1}{2}$}
    };
\end{tikzpicture}\qquad
\begin{tikzpicture}[grow=right, sloped]
\node[bag] {$C^i_{m} $ }
    child {
        node[bag] {$R^{i-1}_{m-1}$}        
            edge from parent 
            node[above] {Closer}
            node[below]  {$1-\frac{\theta}{2}$}
    }
    child {
        node[bag] {$R^{i+1}_{m-1}$}        
        edge from parent         
            node[above] {Further}
            node[below]  {$\frac{\theta}{2}$}
    };
\end{tikzpicture}
\end{center}
To calculate the probability $R^i_m$ that the game lasts $m$ moves  when the robber moves first from position~$i$, we realize that $\frac{1}{2}$ the time he moves closer to the cop and $\frac{1}{2}$ the time he moves further, and the calculation of $R^i_m=\frac{1}{2}C^{i-1}_{m-1}+\frac{1}{2} C^{i+1}_{m-1}$ reduces to calculating the two probabilities $C^{i-1}_{m-1}
$ and $C^{i+1}_{m-1}$, which are the probabilities that the game will last $m-1$ moves provided the cop moves first from positions $i-1$ or $i+1$ respectively.

Similarly, to calculate the probability $C^i_m$ that the game lasts $m$ moves when the cop moves first from position $i$, we realize that she only moves further away from the robber 1/2 of the time when she is moving randomly. Since she moves randomly with probability $\theta$, this means she moves further with probability $\frac{\theta}{2}$. The remaining proportion  ($1- \frac{\theta}{2} $) she moves closer to the robber.  Hence 
$C^i_m=(1-\frac{\theta}{2})R^{i-1}_{m-1}+\frac{\theta}{2} R^{i+1}_{m-1}$ and 
the calculation of $C^i_{m}$ reduces recursively to calculating the probabilities $R^{i-1}_{m-1}$ and $R^{i+1}_{m-1}$ that the game lasts $m-1$ moves when the robber moves first from positions $i-1$ and $i+1$ respectively.

The remaining probabilities when $i=1$ and when $i=\lfloor \frac{v}{2} \rfloor$ follow from similar analyses.
\end{proof}

Next we apply the recursive formulas for $R_m^i$ and $C_m^i$ to calculate recursions for $P_n^i$, where  $n=2m$.  We remark that in small cycles $C_v$ with $v<10$, some of the cases described below collapse. We will consider one such case in detail following the proof of this next result.

\begin{theorem} \label{thm:probability_cycle}Let $v\geq 10$. If $\psi = \frac{1}{2} - \frac{\theta}{4}$,  
then the probability $P_n$ that a game lasts at least $n$ rounds ($n=2m$ moves) on an even length cycle is given by 
\[
P_n^i=\begin{cases}
        \psi P^{i}_{n-1}+\frac{\theta}{4}P^{i+2}_{n-1}, & i= 1 \\
    \frac{1}{2} P^{i}_{n-1}+\frac{\theta}{4}P^{i+2}_{n-1}, & i= 2 \\
    \psi P^{i-2}_{n-1}+\frac{1}{2} P^{i}_{n-1}+\frac{\theta}{4}P^{i+2}_{n-1}, & 2<i< \frac{v}{2}-1\\
    \psi P^{v/2-2}_{n-1}+\frac{1}{2}P^{v/2-1}_{n-1}+\frac{\theta}{4}P^{v/2}_{n-1}, & i=\frac{v}{2}-1 \\
    2\psi P^{v/2-2}_{n-1}+\frac{\theta}{2}P^{v/2}_{n-1}, & i=\frac{v}{2}\\
\end{cases}
\]
and on an odd length cycle 
\[
P_n^i=\begin{cases}
    \psi P^{i}_{n-1}+\frac{\theta}{4}P^{i+2}_{n-1}, & i= 1 \\
    \frac{1}{2} P^{i}_{n-1}+\frac{\theta}{4}P^{i+2}_{n-1}, & i= 2 \\
     \psi P^{i-2}_{n-1}+\frac{1}{2}P^{i}_{n-1}+\frac{\theta}{4}P^{i+2}_{n-1}, & 2<i<\frac{v-3}{2}\\
    \psi P^{i-2}_{n-1}+\frac{1}{2}P^{i}_{n-1}+\frac{\theta}{4}P^{i+1}_{n-1}, & i= \frac{v-3}{2} \\
    \psi P^{i-2}_{n-1} +\psi P^{i-1}_{n-1}+ \frac{\theta}{2} P^{i}_{n-1}, & i= \frac{v-1}{2}.\\
\end{cases}
\]
\end{theorem}

\begin{proof}

The following probability tree describes one round (2 moves) in the general case when the cop and robber are further than 2 moves apart and not located on opposite sides of the cycle.

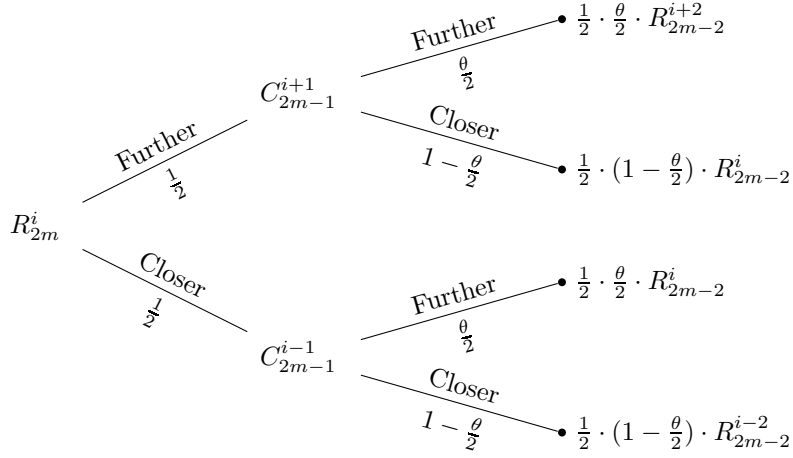
\begin{center}
\tikzstyle{level 1}=[level distance=3.5cm, sibling distance=3.5cm]
\tikzstyle{level 2}=[level distance=3.5cm, sibling distance=2cm]

\tikzstyle{bag} = [text width=4em, text centered]
\tikzstyle{end} = [circle, minimum width=3pt,fill, inner sep=0pt]

\begin{tikzpicture}[grow=right, sloped]
\node[bag] {$R^i_{2m} $ }
    child {
        node[bag] {$C^{i-1}_{2m-1}$}        
            child {
                node[end, label=right:
                    {$\frac{1}{2}\cdot (1- \frac{\theta}{2}) \cdot R^{i-2}_{2m-2}$}] {}
                edge from parent
                node[above] {Closer}
                node[below]  {$1-\frac{\theta}{2}$}
            }
            child {
                node[end, label=right:
                    {{$\frac{1}{2}\cdot\frac{\theta}{2} \cdot R^i_{2m-2}$}}] {}
                edge from parent
                node[above] {Further}
                node[below]  {$\frac{\theta}{2}$}
            }
            edge from parent 
            node[above] {Closer}
            node[below]  {$\frac{1}{2}$}
    }
    child {
        node[bag] {$C^{i+1}_{2m-1}$}        
        child {
                node[end, label=right:
                    {{$\frac{1}{2}\cdot (1-\frac{\theta}{2}) \cdot R^{i}_{2m-2}$}}] {}
                edge from parent
                node[above] {Closer}
                node[below]  {$1-\frac{\theta}{2}$}
            }
            child {
                node[end, label=right:
                    {$\frac{1}{2}\cdot\frac{\theta}{2} \cdot R^{i+2}_{2m-2}$}] {}
                edge from parent
                node[above] {Further}
                node[below]  {$\frac{\theta}{2}$}
            }
        edge from parent         
            node[above] {Further}
            node[below]  {$\frac{1}{2}$}
    };
\end{tikzpicture}
\end{center}

We note that $P^i_{n} = R^{i}_{2m}$ and $P^i_{n-1} = R^{i}_{2m-2}$ since the robber moves first in the game, and we are assuming $n$ rounds is $2m$ moves.  We also note that there are two ways to get from $P^i_{n}$ to $P^{i}_{n-1}$ with combined probability $\frac{\theta}{4} + \frac{1}{2} - \frac{\theta}{4} =\frac{1}{2}$. 
Analyzing the last column of this tree gives the following recursive formula for the probability
$$P^i_n  = \psi P_{n-1}^{i-2} + {\frac{1}{2}} P_{n-1}^i + \frac{\theta}{4} P_{n-1}^{i+2}.   $$

The other cases are obtained by analyzing this same tree when $i=1,2,\lfloor \frac{v}{2} \rfloor -1$, or $\lfloor \frac{v}{2} \rfloor$.   When $i=1,2$ the game can end during the round, so some of the branches in the tree disappear, and when $i=\lfloor \frac{v}{2} \rfloor -1$ or $\lfloor \frac{v}{2} \rfloor$ a player can take a step away from the other player, but their distance remains the same or decreases by one because they are walking around the far side of the cycle.
\end{proof}

 Theorem \ref{thm:probability_cycle} considers large cycle graphs with $v\geq 10$.  To help illustrate the degenerate cases when $i=1,2$, $\lfloor \frac{v}{2} \rfloor-1$ and $\lfloor \frac{v}{2} \rfloor$, we analyze  $\Cycle_5$ where $1 =\lfloor \frac{v}{2} \rfloor-1$ and $2= \lfloor \frac{v}{2} \rfloor$. Hence, on $\Cycle_5$ there are only two starting positions: position 1 is depicted on the left and position 2 is depicted on the right in Figure \ref{cycle5}, and they describe the distance between the cop and robber, which is 1 or 2, respectively.

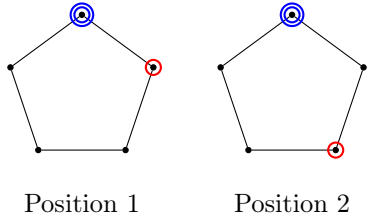
\begin{figure}[h]
\centering
\begin{tikzpicture}
\vertex[fill](a1) at (0,1) {};
\vertex[fill](a2) at (-.95,.30) {};
\vertex[fill](a3) at (.95,.30) {};
\vertex[fill](a4) at (-.58,-.80) {};
\vertex[fill](a5) at (.58,-.80) {};
\draw (a1)--(a2) (a2)--(a4) (a3)--(a5) (a4) -- (a5) (a3)--(a1); 
\draw[blue,thick] (a1) circle (.1cm);
\draw[blue,thick] (a1) circle (.15cm);
\draw[red,thick] (a3) circle (.1cm);
\node at (0,-1.5){Position 1};
\end{tikzpicture}
\hspace{5mm}
\begin{tikzpicture}
\vertex[fill](a1) at (0,1) {};
\vertex[fill](a2) at (-.95,.30) {};
\vertex[fill](a3) at (.95,.30) {};
\vertex[fill](a4) at (-.58,-.80) {};
\vertex[fill](a5) at (.58,-.80) {};
\draw (a1)--(a2) (a2)--(a4) (a3)--(a5) (a4) -- (a5) (a3)--(a1); 
\draw[blue,thick] (a1) circle (.1cm);
\draw[blue,thick] (a1) circle (.15cm);
\draw[red,thick] (a5) circle (.1cm);
\node at (0,-1.5){Position 2};
\end{tikzpicture}
\captionof{figure}{Two starting positions in $\Cycle_5$.}\label{cycle5}
\end{figure}

In this special case, the probabilities $R_m^i$ and $C_m^i$ are related via the following recursive formulas: 
\[
R_m^1   = \frac{1}{2}C_{m-1}^2,\qquad  
R_m^2 = \frac{1}{2}C_{m-1}^1 + \frac{1}{2}   C_{m-1}^2,\qquad
C_m^1  = \frac{\theta}{2}R_{m-1}^2   ,\qquad
C_m^2  =  \left (1- \frac{\theta}{2}  \right)R_{m-1}^1 +  \frac{\theta}{2} R_{m-1}^2.\]

\begin{lemma}\label{lem:rec_cyle}
If $\psi = \frac{1}{2} - \frac{\theta}{4}$, then 
the probability $P_n^i$ that a game starting in position $i$ lasts at least $n$ rounds ($n=2m$ moves) on $\Cycle_5$ is given by 
\begin{align*} 
P_n^1  & = \psi P_{n-1}^1+ \frac{\theta}{4}  P_{n-1}^2, \mbox{ and}\\
P_n^2 & = \psi P_{n-1}^1 + \frac{\theta}{2} P_{n-1}^2  .
\end{align*}
\end{lemma}
\begin{proof}

The following probability trees describe one round (2 moves) when the robber and cop begin the game in positions 1 and 2, respectively.

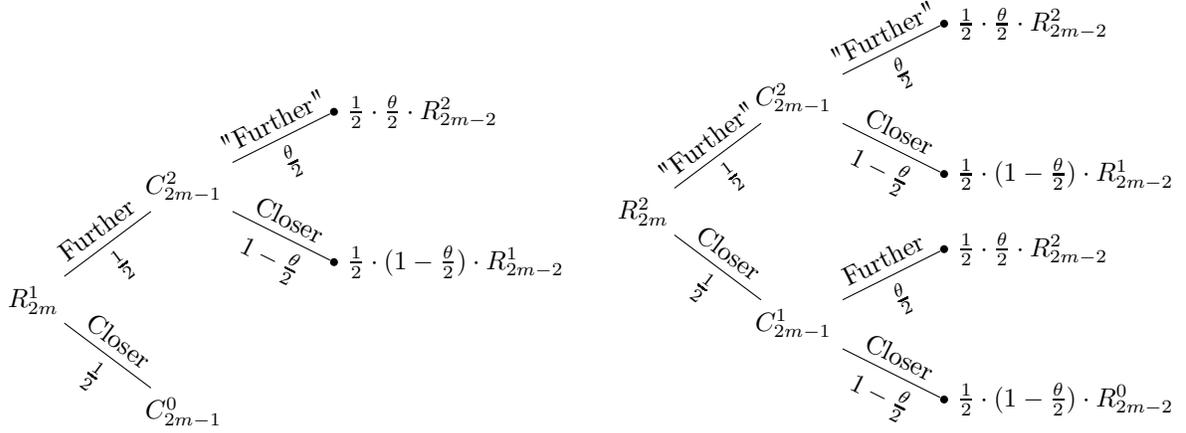
\begin{center}
\tikzstyle{level 1}=[level distance=2cm, sibling distance=3cm]
\tikzstyle{level 2}=[level distance=2cm, sibling distance=2cm]

\tikzstyle{bag} = [text width=4em, text centered]
\tikzstyle{end} = [circle, minimum width=3pt,fill, inner sep=0pt]

\begin{tikzpicture}[grow=right, sloped]
\node[bag] {$R^1_{2m} $ }
    child {
        node[bag] {$C^{0}_{2m-1}$}        
            edge from parent 
            node[above] {Closer}
            node[below]  {$\frac{1}{2}$}
    }
    child {
        node[bag] {$C^{2}_{2m-1}$}        
        child {
                node[end, label=right:
                    {{$\frac{1}{2}\cdot (1-\frac{\theta}{2}) \cdot R^{1}_{2m-2}$}}] {}
                edge from parent
                node[above] {Closer}
                node[below]  {$1-\frac{\theta}{2}$}
            }
            child {
                node[end, label=right:
                    {$\frac{1}{2}\cdot\frac{\theta}{2} \cdot R^{2}_{2m-2}$}] {}
                edge from parent
                node[above] {"Further"}
                node[below]  {$\frac{\theta}{2}$}
            }
        edge from parent         
            node[above] {Further}
            node[below]  {$\frac{1}{2}$}
    };
\end{tikzpicture}
\begin{tikzpicture}[grow=right, sloped]
\node[bag] {$R^2_{2m} $ }
    child {
        node[bag] {$C^{1}_{2m-1}$}        
            child {
                node[end, label=right:
                    {$\frac{1}{2}\cdot (1- \frac{\theta}{2}) \cdot R^{0}_{2m-2}$}] {}
                edge from parent
                node[above] {Closer}
                node[below]  {$1-\frac{\theta}{2}$}
            }
            child {
                node[end, label=right:
                    {{$\frac{1}{2}\cdot\frac{\theta}{2} \cdot R^2_{2m-2}$}}] {}
                edge from parent
                node[above] {Further}
                node[below]  {$\frac{\theta}{2}$}
            }
            edge from parent 
            node[above] {Closer}
            node[below]  {$\frac{1}{2}$}
    }
    child {
        node[bag] {$C^{2}_{2m-1}$}        
        child {
                node[end, label=right:
                    {{$\frac{1}{2}\cdot (1-\frac{\theta}{2}) \cdot R^{1}_{2m-2}$}}] {}
                edge from parent
                node[above] {Closer}
                node[below]  {$1-\frac{\theta}{2}$}
            }
            child {
                node[end, label=right:
                    {$\frac{1}{2}\cdot\frac{\theta}{2} \cdot R^{2}_{2m-2}$}] {}
                edge from parent
                node[above] {"Further"}
                node[below]  {$\frac{\theta}{2}$}
            }
        edge from parent         
            node[above] {"Further"}
            node[below]  {$\frac{1}{2}$}
    };
\end{tikzpicture}
\end{center}
Note that for any $\ell>0$, $C^0_\ell$ and $R^0_\ell$ result in the end of the game.  
Substituting the defined probabilities $P^i_n=R^{i}_{2m}$ and $P^i_{n-1}=R^i_{2m-2}$ into these trees,  we derive the desired recursions \[P^1_{n} = \psi P^1_{n-1} + \frac{\theta}{4} P^2_{n-1}\qquad\mbox{ and }\qquad P^2_n = \psi P_{n-1}^1 + \frac{\theta}{2}P^2_{n-1} .\qedhere\]
\end{proof}

We can now establish the following result. 

\begin{theorem} 
Define the generating functions $P_1(t) = \sum_{n=0}^\infty P_n^1 t^n $ and $P_2(t) = \sum_{n=0}^\infty P_n^2 t^n$, where $P^i_n$ is the probability that the game starting at position $i$ lasts $n$ rounds on $\Cycle_5$.  Then
\begin{align}
    P_1(t) &= \frac{1 - \frac{\theta t}{4}}{ (1-\frac{\theta t}{4})( 2-\psi t)-1} \\
    P_2(t) & = \frac{1 }{ (1-\frac{\theta t}{4})( 2-\psi t)-1}.
\end{align}
\end{theorem}

\begin{proof} Given the recursive formulas in Lemma \ref{lem:rec_cyle} subtracting $P_n^1-P_n^2$ and $P_n^2-2P_n^1$ gives
\begin{align*} 
P_n^2-P_n^1  & =  \frac{\theta}{4}P_{n-1}^2 \\
P_n^2-2P_n^1 & = -\psi P_{n-1}^1. 
\end{align*}

Multiplying both equations by $t^{n}$ and summing over $n\geq 1$ gives the following system of equations for the generating functions $P_1(t)$ and $P_2(t)$:
\begin{align*} 
\sum_{n\geq 1}P_n^2t^n-\sum_{n\geq 1}P_n^1t^n  & =  \sum_{n\geq 1}\frac{\theta}{4}P_{n-1}^2t^n\\
\left[\left(\sum_{n\geq 0}P_n^2t^n\right)-P_0^2\right]-\left[\left(\sum_{n\geq 0}P_n^1t^n\right)-P_0^1\right]  & =  \frac{\theta t}{4}\sum_{n\geq 1}P_{n-1}^2t^{n-1}\\
\left[P_2(t)-P_0^2\right]-\left[P_1(t)-P_0^1\right]  & =  \frac{\theta t}{4}P_2(t)\\
P_2(t)-P_1(t) & =  \frac{\theta t}{4}P_2(t),\mbox{ since $P_0^1=P_0^2=1$}
\end{align*}
and 
\begin{align*} 
\sum_{n\geq 1}P_n^2t^n-2\sum_{n\geq 1}P_n^1t^n  & =  -\sum_{n\geq 1}\psi P_{n-1}^1t^n\\
\left[\left(\sum_{n\geq 1}P_n^2t^n\right)-P_0^2\right] -2\left[\left(\sum_{n\geq 0}P_n^1t^n\right)-P_0^1\right]  & =  -\psi t\sum_{n\geq 1} P_{n-1}^1t^{n-1}\\
\left[P_2(t)-P_0^2\right] -2\left[P_1(t)-P_0^1\right]  & =  -\psi tP_1(t)\\
P_2(t) -2P_1(t) & =  -\psi tP_1(t)-1,\mbox{ since $P_0^1=P_0^2=1$.}
\end{align*}
Thus 
\begin{align*} 
P_2(t)-P_1(t)  & = \frac{\theta t}{4}P_2(t) \\
P_2(t)-2P_1(t) & = -\psi t P_1(t)-1.
\end{align*}
Solving this system completes the proof.
\end{proof}

Since evaluating these generating functions at $t=1$ gives the expected number of rounds, we have the following corollary.

\begin{corollary}
Let $X$ denote the random variable giving the capture time and let $\EE^i[X]$ denote the expected encounter time between the robber and cop beginning the game in position $i$ on $\Cycle_5$.  If $\theta$ denotes the proportion of the cop's movement that is random, then the expected encounter time between the robber and cop is given by
\begin{align}
    \EE^1[X]= & \frac{1 - \frac{\theta}{4}}{ (1-\frac{\theta}{4})( 2-\psi)-1},\mbox{ and}\\
    \EE^2[X]= & \frac{1 }{ (1-\frac{\theta}{4})( 2-\psi)-1}.
\end{align}

\end{corollary}

\section{A different probability question}\label{sec:new probability}
The technique presented in the previous section also lends itself well when considering a variant of the question of interest:
\begin{center}
    \emph{What is the probability that the robber remains free at time 
$m$ when\\ the cop and robber end in Position $i$ on graph $G$ at time $m$?
}
\end{center}

We now answer this question in the case of friendship graphs.

\subsection{Friendship graphs}\label{sec:friendship}
A friendship graph, denoted by $F_k$, is a planar undirected graph consisting of $k$ triangular cliques sharing one common vertex. Hence $F_k$ has $2k + 1$ vertices and $3k$ edges. Figure \ref{fig:samplefriendshipgraphs} gives the friendship graphs $F_2$, $F_3$, and $F_4$.
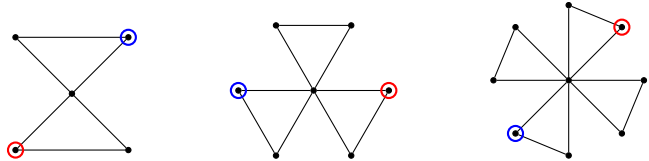
\begin{figure}[h]
\centering
\begin{tikzpicture}
\vertex[fill](a1) at (.75,.75) {};
\vertex[fill](a2) at (-.75,.75) {};
\vertex[fill](a3) at (-.75,-.75) {};
\vertex[fill](a4) at (.75,-.75) {};
\vertex[fill](a5) at (0,0) {};
\draw (a1)--(a2) ;
\draw (a2)--(a4);
\draw (a3)--(a1);
\draw (a3)--(a4) ;
\draw[blue,thick] (a1) circle (.1cm);
\draw[red,thick] (a3) circle (.1cm);
\end{tikzpicture}
\hspace{1cm}
\begin{tikzpicture}
\vertex[fill](a1) at (.5000000000, .8660254040) {};
\vertex[fill](a2) at (-.5000000000, .8660254040) {};
\vertex[fill](a3) at (-1., 0.) {};
\vertex[fill](a4) at (-.5000000000, -.8660254040) {};
\vertex[fill](a5) at (.5000000000, -.8660254040) {};
\vertex[fill](a6) at (1., 0.) {};
\vertex[fill](a0) at (0., 0.) {};
\draw (a1)--(a2) ;
\draw (a0)--(a2) ;
\draw (a3)--(a4) ;
\draw (a0)--(a4) ;
\draw (a3)--(a0) ;
\draw (a0)--(a1) ;
\draw (a5)--(a6);
\draw (a5)--(a0);
\draw (a0)--(a6);
\draw[blue,thick] (-1,0) circle (.1cm);
\draw[red,thick] (1,0) circle (.1cm);
\end{tikzpicture}
\hspace{1cm}
\begin{tikzpicture}
\vertex[fill](a0) at (0., 0.) {};
\vertex[fill](a1) at (.707107, .707107) {};
\vertex[fill](a2) at (0,1) {};
\vertex[fill](a3) at (-0.707107,0.707107) {};
\vertex[fill](a4) at (-1,0) {};
\vertex[fill](a5) at (-0.707107,-0.707107) {};
\vertex[fill](a6) at (0,-1) {};
\vertex[fill](a7) at (0.707107,-0.707107) {};
\vertex[fill](a8) at (1,0) {};
\draw (a0)--(a1);
\draw (a0)--(a2);
\draw (a0)--(a3);
\draw (a0)--(a4);
\draw (a0)--(a5);
\draw (a0)--(a6);
\draw (a0)--(a7);
\draw (a0)--(a8);
\draw (a1)--(a2);
\draw (a3)--(a4);
\draw (a5)--(a6);
\draw (a7)--(a8);
\draw[blue,thick] (-0.707107,-0.707107) circle (.1cm);
\draw[red,thick] (0.707107,0.707107) circle (.1cm);
\end{tikzpicture}
\captionof{figure}{Examples of friendship graphs ($F_k$ for $k=2,3,4$) with cop (blue) and robber (red).}
    \label{fig:samplefriendshipgraphs}
\end{figure}
There are always only four  possible configurations for the cop and robber (up to isomorphism) on the friendship graph. Specifically, the cop and robber can both be on the outside edge of the same component, one or the other can be at the central vertex, or they can be on separate components. We illustrate these cases in Positions \ref{fset1}-\ref{fset4} below.\\

\makeatletter 
\renewcommand{\theposition}{\@arabic\c@position}
\makeatother
\setcounter{position}{0}

\begin{minipage}{\textwidth}
\begin{minipage}[b]{0.2\textwidth}
\centering
\resizebox{.5\textwidth}{!}{
\begin{tikzpicture}
\vertex[fill](a0) at (0,0) {};
\vertex[fill](b1) at (-1,.5) {};
\vertex[fill](b2) at (-1,-.5) {};
\vertex[fill](c1) at (1,.5) {};
\vertex[fill](c2) at (1,-.5) {};
\draw (b1)--(a0) (b1)--(b2) (b2)--(a0);
\draw (c1)--(a0) (c1)--(c2) (c2)--(a0);
\draw[blue,thick] (-1,.5) circle (.1cm);
\draw[red,thick] (-1,-.5) circle (.1cm);
\end{tikzpicture}}
\captionof{position}{}\label{fset1}
\end{minipage}
\hspace{5mm}
\begin{minipage}[b]{0.2\textwidth}
\centering
\resizebox{.5\textwidth}{!}{
\begin{tikzpicture}
\vertex[fill](a0) at (0,0) {};
\vertex[fill](b1) at (-1,.5) {};
\vertex[fill](b2) at (-1,-.5) {};
\vertex[fill](c1) at (1,.5) {};
\vertex[fill](c2) at (1,-.5) {};
\draw (b1)--(a0) (b1)--(b2) (b2)--(a0);
\draw (c1)--(a0) (c1)--(c2) (c2)--(a0);
\draw[blue,thick] (-1,.5) circle (.1cm);
\draw[red,thick] (1,.5) circle (.1cm);
\end{tikzpicture}}
\captionof{position}{}\label{fset2}
\end{minipage}
\hspace{5mm}
\begin{minipage}[b]{0.2\textwidth}
\centering
\resizebox{.5\textwidth}{!}{
\begin{tikzpicture}
\vertex[fill](a0) at (0,0) {};
\vertex[fill](b1) at (-1,.5) {};
\vertex[fill](b2) at (-1,-.5) {};
\vertex[fill](c1) at (1,.5) {};
\vertex[fill](c2) at (1,-.5) {};
\draw (b1)--(a0) (b1)--(b2) (b2)--(a0);
\draw (c1)--(a0) (c1)--(c2) (c2)--(a0);
\draw[blue,thick] (-1,.5) circle (.1cm);
\draw[red,thick] (0,0) circle (.1cm);
\end{tikzpicture}}
\captionof{position}{}\label{fset3}
\end{minipage}
\hspace{5mm}
\begin{minipage}[b]{0.2\textwidth}
\centering
\resizebox{.5\textwidth}{!}{
\begin{tikzpicture}
\vertex[fill](a0) at (0,0) {};
\vertex[fill](b1) at (-1,.5) {};
\vertex[fill](b2) at (-1,-.5) {};
\vertex[fill](c1) at (1,.5) {};
\vertex[fill](c2) at (1,-.5) {};
\draw (b1)--(a0) (b1)--(b2) (b2)--(a0);
\draw (c1)--(a0) (c1)--(c2) (c2)--(a0);
\draw[red,thick] (-1,.5) circle (.1cm);
\draw[blue,thick] (0,0) circle (.1cm);
\end{tikzpicture}}
\captionof{position}{}\label{fset4}
\end{minipage}
\end{minipage}

\vspace{12pt}

We let $\PP_m^i(F_k)$ with $1 \leq i \leq 4$ denote the probability that the robber is still free after $m$ time steps when the cop and robber end in Position $i$ at time $m$ on graph $F_k$.  The remainder of this section is dedicated to calculating the values of $\PP_m^i(F_k)$ with $1 \leq i \leq 4$.

\begin{theorem}\label{thm:FriendshipRecursive}Let $1\leq i\leq 4$. If $\theta$ denotes the proportion of the cop's movement that is random, then the values of $\PP_m^i (F_k)$ are described by the following recursive formulas.\\
For odd values of $m$
\begin{align}
\PP^1_m(F_k)&=\left(\frac{1}{2k}\right)\theta \PP^4_{m-1}(F_k)\\
\PP^2_m(F_k)&=\left(\frac{1}{2}\right)\theta \PP^2_{m-1}(F_k)+\left(\frac{2k-2}{2k}\right)\theta \PP^4_{m-1}(F_k)\\
\PP^3_m(F_k)&=\left(\frac{1}{2}\right)\theta \PP^3_{m-1}(F_k)\\
\PP^4_m(F_k)&=\left(\frac{1}{2}\right)\theta \PP^1_{m-1}(F_k)+\left(1-\frac{1}{2}\theta\right)\PP^2_{m-1}(F_k)
\end{align}
and for even values of $m$
\begin{align}
\PP^1_m(F_k)&=\left(\frac{1}{2k}\right)\PP^3_{m-1}(F_k)\\
\PP^2_m(F_k)&=\left(\frac{1}{2}\right) \PP^2_{m-1}(F_k)+\left(\frac{2k-2}{2k}\right)\PP^3_{m-1}(F_k).\\
\PP^3_m(F_k)&=\frac{1}{2}\left(\PP^1_{m-1}(F_k)+\frac{1}{2}\PP^2_{m-1}(F_k)\right)\\
\PP^4_m(F_k)&=\frac{1}{2}\PP^4_{m-1}(F_k).
\end{align}
\end{theorem}
\begin{proof}

We begin by considering the case where it is the robber's turn to move (that is the case where $m$ is odd).  We proceed by case analysis on whether $m$ is odd or even.\\

\noindent
{\bf Case 1 ($m$ odd):} There are four positions we must consider individually. However, in all of these cases as $m$ is odd, we note at time $m-1$ the cop made a move. We proceed with a case-by-case analysis.
\begin{enumerate}

\item[1.] At time $m$ the cop and robber are in Position \ref{fset1}. Figure \ref{odd-conf1} gives the possible location where the cop was at time $m-1$ given that at time $m$ the cop and robber are in Position \ref{fset1}.  Hence, the probability that the robber remains free at time $m$ is given by the product of two probabilities: The first is the probability that the robber was free at time $m-1$ in Position \ref{fset4}, which is given by $\PP^4_{m-1}(F_k)$. The second probability is the probability that the cop did not catch the robber in her move at time $m-1$, which means that the cop moved randomly with probability $\theta$ times the probability of selecting one of the $2k$ available vertices, so that the cop and robber are in Position \ref{fset1} at time $m$. Thus $\PP^1_{m}(F_k)=\left(\frac{1}{2k}\right)\theta \PP^4_{m-1}(F_k)$.

\begin{figure}[h]
\centering
\begin{tikzpicture}
\node at (0,1.5) {time $m-1$};
\vertex[fill,label=above:{\small C}](a0) at (0,0) {};
\vertex[fill,label=above:{\small R}](b1) at (-1,.5) {};
\vertex[fill](b2) at (-1,-.5) {};
\vertex[fill](c1) at (1,.5) {};
\vertex[fill](c2) at (1,-.5) {};
\draw (b1)--(a0) (b1)--(b2) (b2)--(a0);
\draw (c1)--(a0) (c1)--(c2) (c2)--(a0);
\draw[red,thick] (-1,.5) circle (.1cm);
\draw[blue,thick] (0,0) circle (.1cm);
\node at (4,1.5) {time $m$};
\vertex[fill](a01) at (4,0) {};
\vertex[fill,label=above:{\small C}](b11) at (3,.5) {};
\vertex[fill,label=below:{\small R}](b21) at (3,-.5) {};
\vertex[fill](c11) at (5,.5) {};
\vertex[fill](c21) at (5,-.5) {};
\draw (b11)--(a01) (b11)--(b21) (b21)--(a01);
\draw (c11)--(a01) (c11)--(c21) (c21)--(a01);
\draw[blue,thick] (3,.5) circle (.1cm);
\draw[red,thick] (3,-.5) circle (.1cm);
\draw[thick,->] (1.5,0) -- (2.5,0) node[anchor=north west] {};
\end{tikzpicture}
\caption{Position \ref{fset4} at even time $m-1$ to Position \ref{fset1} at odd time $m$.}\label{odd-conf1}
\end{figure}
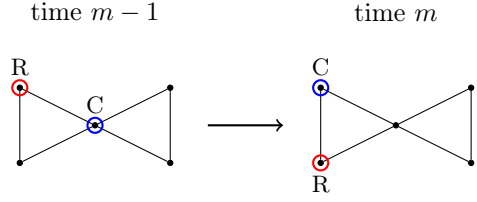

\item[2.] At time $m$ the cop and robber are in Position \ref{fset2}. Figure \ref{odd-conf2} gives the possible locations where the cop was at time $m-1$ given that at time $m$ the cop and robber are in Position \ref{fset2}.  Hence, the probability that the robber remains free at time $m$ is given by the product of two probabilities for each of the possible positions at time $m-1$. \\

If at time $m-1$ the cop and robber were at Position \ref{fset2}, the first probability in the first product is the probability that the robber was free at time $m-1$ in Position \ref{fset2}, which is given by $\PP^2_{m-1}(F_k)$. The second probability in the first product is the probability that the cop did not catch the robber in their move at time $m-1$, which means that the cop moved randomly with probability $\theta$ times the probability of moving to the other vertex in the same component, which yields $\frac{1}{2}\theta$. Thus the first probability is given by $\left(\frac{1}{2}\right)\theta \PP^2_{m-1}(F_k)$. \\

For the second product we begin with Position \ref{fset4} at time $m-1$. The first probability in the second product is the probability that the robber was free at time $m-1$ in Position \ref{fset4}, which is given by $\PP^4_{m-1}(F_k)$. The second probability in the second product is the probability that the cop did not catch the robber in his move at time $m-1$, which means that the cop moved randomly with probability $\theta$ times the probability of moving to any other vertex on the graph with the exception of the available vertex in the robber's component, which yields $\frac{2k-2}{2k}\theta$. Thus the second probability is given by $\left(\frac{2k-2}{2k}\right)\theta \PP^4_{m-1}(F_k)$. \\

Finally we observe that at time $m-1$ only one of the two possible positions can occur, therefore
\[\PP^2_m(F_k)=\left(\frac{1}{2}\right)\theta \PP^2_{m-1}(F_k)+\left(\frac{2k-2}{2k}\right)\theta \PP^4_{m-1}(F_k).\]

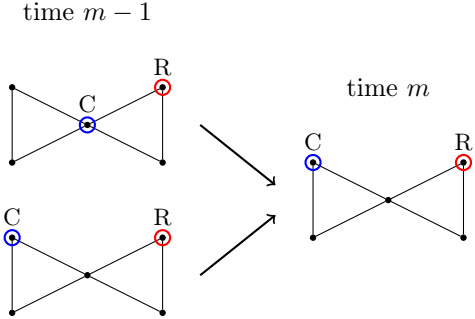
\begin{figure}[h]
\centering
\begin{tikzpicture}
\node at (0,2.5) {time $m-1$};
\vertex[fill,label=above:{\small C}](a0) at (0,1) {};
\vertex[fill](b1) at (-1,1.5) {};
\vertex[fill](b2) at (-1,.5) {};
\vertex[fill,label=above:{\small R}](c1) at (1,1.5) {};
\vertex[fill](c2) at (1,.5) {};
\draw (b1)--(a0) (b1)--(b2) (b2)--(a0);
\draw (c1)--(a0) (c1)--(c2) (c2)--(a0);
\draw[blue,thick] (a0) circle (.1cm);
\draw[red,thick] (c1) circle (.1cm);
\vertex[fill](a02) at (0,-1) {};
\vertex[fill,label=above:{\small C}](b12) at (-1,-.5) {};
\vertex[fill](b22) at (-1,-1.5) {};
\vertex[fill,label=above:{\small R}](c12) at (1,-.5) {};
\vertex[fill](c22) at (1,-1.5) {};
\draw (b12)--(a02) (b12)--(b22) (b22)--(a02);
\draw (c12)--(a02) (c12)--(c22) (c22)--(a02);
\draw[blue,thick] (b12) circle (.1cm);
\draw[red,thick] (c12) circle (.1cm);
\node at (4,1.5) {time $m$};
\vertex[fill](a01) at (4,0) {};
\vertex[fill,label=above:{\small C}](b11) at (3,.5) {};
\vertex[fill](b21) at (3,-.5) {};
\vertex[fill,label=above:{\small R}](c11) at (5,.5) {};
\vertex[fill](c21) at (5,-.5) {};
\draw (b11)--(a01) (b11)--(b21) (b21)--(a01);
\draw (c11)--(a01) (c11)--(c21) (c21)--(a01);
\draw[blue,thick] (b11) circle (.1cm);
\draw[red,thick] (c11) circle (.1cm);
\draw[thick,->] (1.5,1) -- (2.5,.2) node[anchor=north west] {};
\draw[thick,->] (1.5,-1) -- (2.5,-.2) node[anchor=north west] {};
\end{tikzpicture}
\caption{Positions \ref{fset2} and \ref{fset4} at even time $m-1$ to Position \ref{fset2} at odd time $m$.}\label{odd-conf2}
\end{figure}

\item[3.] At time $m$ the cop and robber are in Position \ref{fset3}. Figure \ref{odd-conf3} gives the possible location where the cop was at time $m-1$ given that at time $m$ the cop and robber are in Position \ref{fset3}.  Hence, the probability that the robber remains free at time $m$ is given by the product of two probabilities: The first is the probability that the robber was free at time $m-1$ in Position \ref{fset3}, which is given by $\PP^3_{m-1}(F_k)$. The second probability is the probability that the cop did not catch the robber in her move at time $m-1$, which means that the cop moved randomly with probability $\theta$ times the probability of moving to the other vertex in the same component, which yields $\frac{1}{2}\theta$. Thus $\PP^3_{m}(F_k)=\left(\frac{1}{2}\right)\theta \PP^3_{m-1}(F_k)$.

\begin{figure}[h]
\centering
\begin{tikzpicture}
\node at (0,1.5) {time $m-1$};
\vertex[fill,label=above:{\small R}](a0) at (0,0) {};
\vertex[fill,label=above:{\small C}](b1) at (-1,.5) {};
\vertex[fill](b2) at (-1,-.5) {};
\vertex[fill](c1) at (1,.5) {};
\vertex[fill](c2) at (1,-.5) {};
\draw (b1)--(a0) (b1)--(b2) (b2)--(a0);
\draw (c1)--(a0) (c1)--(c2) (c2)--(a0);
\draw[blue,thick] (b1) circle (.1cm);
\draw[red,thick] (a0) circle (.1cm);
\node at (4,1.5) {time $m$};
\vertex[fill,label=below:{\small R}](a01) at (4,0) {};
\vertex[fill,label=above:{\small C}](b11) at (3,.5) {};
\vertex[fill](b21) at (3,-.5) {};
\vertex[fill](c11) at (5,.5) {};
\vertex[fill](c21) at (5,-.5) {};
\draw (b11)--(a01) (b11)--(b21) (b21)--(a01);
\draw (c11)--(a01) (c11)--(c21) (c21)--(a01);
\draw[blue,thick] (b11) circle (.1cm);
\draw[red,thick] (a01) circle (.1cm);
\draw[thick,->] (1.5,0) -- (2.5,0) node[anchor=north west] {};
\end{tikzpicture}
\caption{Position \ref{fset3} at even time $m-1$ to Position \ref{fset3} at odd time $m$.}\label{odd-conf3}
\end{figure}
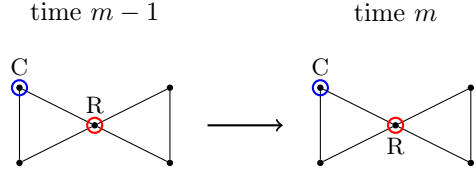

\item[4.] At time $m$ the cop and robber are in Position \ref{fset4}. Figure \ref{odd-conf4} gives the possible locations where the cop was at time $m-1$ given that at time $m$ the cop and robber are in Position \ref{fset4}.  Hence, the probability that the robber remains free at time $m$ is given by the product of two probabilities for each of the possible positions at time $m-1$. \\

If at time $m-1$ the cop and robber were at Position \ref{fset1}, the first probability in the first product is the probability that the robber was free at time $m-1$ in Position \ref{fset1}, which is given by $\PP^1_{m-1}(F_k)$. The second probability in the first product is the probability that the cop did not catch the robber in his move at time $m-1$, which means that the cop moved randomly with probability $\theta$ times the probability of moving to the center vertex of the graph, which yields $\frac{1}{2}\theta$. Thus the first probability is given by $\left(\frac{1}{2}\right)\theta \PP^1_{m-1}(F_k)$. \\

For the second product we begin with Position \ref{fset2} at time $m-1$.The first probability in the second product is the probability that the robber was free at time $m-1$ in Position \ref{fset2}, which is given by $\PP^2_{m-1}(F_k)$. The second probability in the second product is the probability that the cop did not catch the robber in their move at time $m-1$, which means that the cop moved either randomly with probability $\theta$ times the probability of moving to the center vertex on the graph, or the probability of moving toward the robber (since the shortest path between the cop and the robber goes through the center vertex of the graph), which yields $\frac{1}{2}\theta+(1-\theta)$. Thus the second probability is given by $(1-\frac{1}{2}\theta)\PP^2_{m-1}(F_k)$. \\

Finally we observe that at time $m-1$ only one of the two possible positions can occur, therefore
\[\PP^4_m(F_k)=\left(\frac{1}{2}\right)\theta \PP^1_{m-1}(F_k)+\left(1-\frac{1}{2}\theta\right)\PP^2_{m-1}(F_k).\]

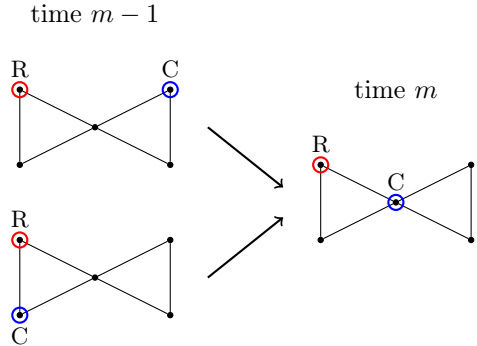
\begin{figure}[h!]
\centering
\begin{tikzpicture}
\node at (0,2.5) {time $m-1$};
\vertex[fill](a0) at (0,1) {};
\vertex[fill,label=above:{\small R}](b1) at (-1,1.5) {};
\vertex[fill](b2) at (-1,.5) {};
\vertex[fill,label=above:{\small C}](c1) at (1,1.5) {};
\vertex[fill](c2) at (1,.5) {};
\draw (b1)--(a0) (b1)--(b2) (b2)--(a0);
\draw (c1)--(a0) (c1)--(c2) (c2)--(a0);
\draw[blue,thick] (c1) circle (.1cm);
\draw[red,thick] (b1) circle (.1cm);
\vertex[fill](a02) at (0,-1) {};
\vertex[fill,label=above:{\small R}](b12) at (-1,-.5) {};
\vertex[fill,label=below:{\small C}](b22) at (-1,-1.5) {};
\vertex[fill](c12) at (1,-.5) {};
\vertex[fill](c22) at (1,-1.5) {};
\draw (b12)--(a02) (b12)--(b22) (b22)--(a02);
\draw (c12)--(a02) (c12)--(c22) (c22)--(a02);
\draw[blue,thick] (b22) circle (.1cm);
\draw[red,thick] (b12) circle (.1cm);
\node at (4,1.5) {time $m$};
\vertex[fill,label=above:{\small C}](a01) at (4,0) {};
\vertex[fill,label=above:{\small R}](b11) at (3,.5) {};
\vertex[fill](b21) at (3,-.5) {};
\vertex[fill](c11) at (5,.5) {};
\vertex[fill](c21) at (5,-.5) {};
\draw (b11)--(a01) (b11)--(b21) (b21)--(a01);
\draw (c11)--(a01) (c11)--(c21) (c21)--(a01);
\draw[blue,thick] (a01) circle (.1cm);
\draw[red,thick] (b11) circle (.1cm);
\draw[thick,->] (1.5,1) -- (2.5,.2) node[anchor=north west] {};
\draw[thick,->] (1.5,-1) -- (2.5,-.2) node[anchor=north west] {};
\end{tikzpicture}
\caption{Positions \ref{fset1} and \ref{fset2} at even time $m-1$ to Position \ref{fset4} at odd time $m$.}\label{odd-conf4}
\end{figure}

\end{enumerate}

\noindent
{\bf Case 2 ($m$ even):} There are four positions we must consider individually. However, in all of these cases as $m$ is even, we note at time $m-1$ the robber made a move. We proceed with a case-by-case analysis.
\begin{enumerate}

\item[1.] At time $m$ the cop and robber are in Position \ref{fset1}. Figure \ref{even-conf1} gives the possible location where the robber was at time $m-1$ given that at time $m$ the cop and robber are in Position \ref{fset1}.  Hence, the probability that the robber remains free at time $m$ is given by the product of two probabilities: The first is the probability that the robber was free at time $m-1$ in Position \ref{fset3}, which is given by $\PP^3_{m-1}(F_k)$. The second probability is the probability that the robber did not move to the vertex occupied by the cop in his move at time $m-1$, which means that the robber moved to the other vertex in the same component as the cop, so that the cop and robber are in Position \ref{fset1} at time $m$. Thus $\PP^1_{m}(F_k)=\left(\frac{1}{2k}\right) \PP^3_{m-1}(F_k)$.

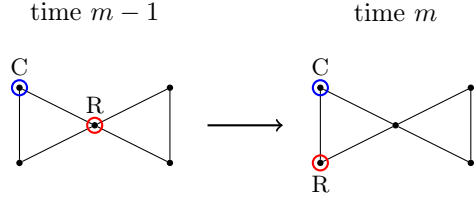
\begin{figure}[h]
\centering
\begin{tikzpicture}
\node at (0,1.5) {time $m-1$};
\vertex[fill,label=above:{\small R}](a0) at (0,0) {};
\vertex[fill,label=above:{\small C}](b1) at (-1,.5) {};
\vertex[fill](b2) at (-1,-.5) {};
\vertex[fill](c1) at (1,.5) {};
\vertex[fill](c2) at (1,-.5) {};
\draw (b1)--(a0) (b1)--(b2) (b2)--(a0);
\draw (c1)--(a0) (c1)--(c2) (c2)--(a0);
\draw[blue,thick] (b1) circle (.1cm);
\draw[red,thick] (a0) circle (.1cm);
\node at (4,1.5) {time $m$};
\vertex[fill](a01) at (4,0) {};
\vertex[fill,label=above:{\small C}](b11) at (3,.5) {};
\vertex[fill,label=below:{\small R}](b21) at (3,-.5) {};
\vertex[fill](c11) at (5,.5) {};
\vertex[fill](c21) at (5,-.5) {};
\draw (b11)--(a01) (b11)--(b21) (b21)--(a01);
\draw (c11)--(a01) (c11)--(c21) (c21)--(a01);
\draw[blue,thick] (b11) circle (.1cm);
\draw[red,thick] (b21) circle (.1cm);
\draw[thick,->] (1.5,0) -- (2.5,0) node[anchor=north west] {};
\end{tikzpicture}
\caption{Position \ref{fset1} at even time $m$.}\label{even-conf1}
\end{figure}

\item[2.] At time $m$ the cop and robber are in Position \ref{fset2}. Figure \ref{even-conf2} gives the possible locations where the robber was at time $m-1$ given that at time $m$ the cop and robber are in Position \ref{fset2}.  Hence, the probability that the robber remains free at time $m$ is given by the product of two probabilities for each of the possible positions at time $m-1$. 

If at time $m-1$ the cop and robber were at Position \ref{fset2}, the first probability in the first product is the probability that the robber was free at time $m-1$ in Position \ref{fset2}, which is given by $\PP^2_{m-1}(F_k)$. The second probability in the first product is the probability that the robber did not move to the vertex occupied by the cop in his move at time $m-1$, which means that the robber moved to the other vertex in the same component, which yields $\frac{1}{2}$. Thus the first probability is given by $\left(\frac{1}{2}\right) \PP^2_{m-1}(F_k)$.

For the second product we begin with Position \ref{fset3} at time $m-1$. The first probability in the second product is the probability that the robber was free at time $m-1$ in Position \ref{fset3}, which is given by $\PP^3_{m-1}(F_k)$. The second probability in the second product is the probability that the robber did not move to the vertex occupied by the cop in his move at time $m-1$, which means that the robber moved to any of the other $2k-2$ vertex from $2k$ neighboring vertices. Thus the second probability is given by $\left(\frac{2k-2}{2k}\right)\PP^3_{m-1}(F_k)$. 

Finally we observe that at time $m-1$ only one of the two possible positions can occur, therefore
\[\PP^2_m(F_k)=\left(\frac{1}{2}\right) \PP^2_{m-1}(F_k)+\left(\frac{2k-2}{2k}\right)\PP^3_{m-1}(F_k).\]

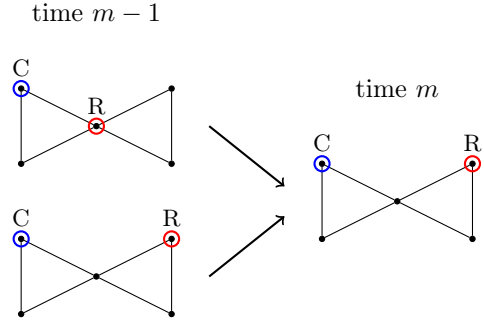
\begin{figure}[h!]
\centering
\begin{tikzpicture}
\node at (0,2.5) {time $m-1$};
\vertex[fill,label=above:{\small R}](a0) at (0,1) {};
\vertex[fill,label=above:{\small C}](b1) at (-1,1.5) {};
\vertex[fill](b2) at (-1,.5) {};
\vertex[fill](c1) at (1,1.5) {};
\vertex[fill](c2) at (1,.5) {};
\draw (b1)--(a0) (b1)--(b2) (b2)--(a0);
\draw (c1)--(a0) (c1)--(c2) (c2)--(a0);
\draw[blue,thick] (b1) circle (.1cm);
\draw[red,thick] (a0) circle (.1cm);
\vertex[fill](a02) at (0,-1) {};
\vertex[fill,label=above:{\small C}](b12) at (-1,-.5) {};
\vertex[fill](b22) at (-1,-1.5) {};
\vertex[fill,label=above:{\small R}](c12) at (1,-.5) {};
\vertex[fill](c22) at (1,-1.5) {};
\draw (b12)--(a02) (b12)--(b22) (b22)--(a02);
\draw (c12)--(a02) (c12)--(c22) (c22)--(a02);
\draw[blue,thick] (b12) circle (.1cm);
\draw[red,thick] (c12) circle (.1cm);
\node at (4,1.5) {time $m$};
\vertex[fill](a01) at (4,0) {};
\vertex[fill,label=above:{\small C}](b11) at (3,.5) {};
\vertex[fill](b21) at (3,-.5) {};
\vertex[fill,label=above:{\small R}](c11) at (5,.5) {};
\vertex[fill](c21) at (5,-.5) {};
\draw (b11)--(a01) (b11)--(b21) (b21)--(a01);
\draw (c11)--(a01) (c11)--(c21) (c21)--(a01);
\draw[blue,thick] (b11) circle (.1cm);
\draw[red,thick] (c11) circle (.1cm);
\draw[thick,->] (1.5,1) -- (2.5,.2) node[anchor=north west] {};
\draw[thick,->] (1.5,-1) -- (2.5,-.2) node[anchor=north west] {};
\end{tikzpicture}
\caption{Position \ref{fset2} at even time $m$ .}\label{even-conf2}
\end{figure}

\item[3.] At time $m$ the cop and robber are in Position \ref{fset3}. Figure \ref{even-conf3} gives the possible locations where the robber was at time $m-1$ given that at time $m$ the cop and robber are in Position \ref{fset3}.  Hence, the probability that the robber remains free at time $m$ is given by the product of two probabilities for each of the possible positions at time $m-1$. 

If at time $m-1$ the cop and robber were at Position \ref{fset1}, the first probability in the first product is the probability that the robber was free at time $m-1$ in Position \ref{fset1}, which is given by $\PP^1_{m-1}(F_k)$. The second probability in the first product is the probability that the robber did not move to the vertex occupied by the cop in his move at time $m-1$, which means that the robber moved to center vertex of the graph, which yields $\frac{1}{2}$. Thus the first probability is given by $\left(\frac{1}{2}\right) \PP^1_{m-1}(F_k)$.

For the second product we begin with Position \ref{fset2} at time $m-1$. The first probability in the second product is the probability that the robber was free at time $m-1$ in Position \ref{fset2}, which is given by $\PP^2_{m-1}(F_k)$. The second probability in the second product is the probability that the robber moved to the center vertex of the graph, this occurs with probability $\frac{1}{2}$. Thus the second probability is given by $\left(\frac{1}{2}\right)\PP^2_{m-1}(F_k)$. 

Finally we observe that at time $m-1$ only one of the two possible positions can occur, therefore
\[\PP^3_m(F_k)=\left(\frac{1}{2}\right) \PP^1_{m-1}(F_k)+\left(\frac{1}{2}\right)\PP^2_{m-1}(F_k).\]

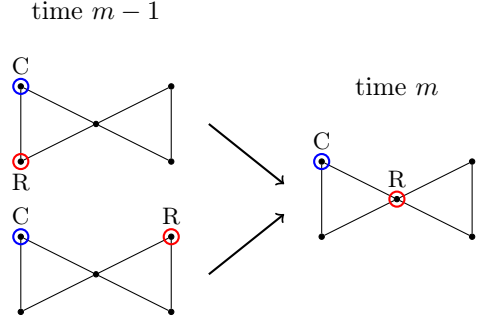
\begin{figure}[h!]
\centering
\begin{tikzpicture}
\node at (0,2.5) {time $m-1$};
\vertex[fill](a0) at (0,1) {};
\vertex[fill,label=above:{\small C}](b1) at (-1,1.5) {};
\vertex[fill,label=below:{\small R}](b2) at (-1,.5) {};
\vertex[fill](c1) at (1,1.5) {};
\vertex[fill](c2) at (1,.5) {};
\draw (b1)--(a0) (b1)--(b2) (b2)--(a0);
\draw (c1)--(a0) (c1)--(c2) (c2)--(a0);
\draw[blue,thick] (b1) circle (.1cm);
\draw[red,thick] (b2) circle (.1cm);
\vertex[fill](a02) at (0,-1) {};
\vertex[fill,label=above:{\small C}](b12) at (-1,-.5) {};
\vertex[fill](b22) at (-1,-1.5) {};
\vertex[fill,label=above:{\small R}](c12) at (1,-.5) {};
\vertex[fill](c22) at (1,-1.5) {};
\draw (b12)--(a02) (b12)--(b22) (b22)--(a02);
\draw (c12)--(a02) (c12)--(c22) (c22)--(a02);
\draw[blue,thick] (b12) circle (.1cm);
\draw[red,thick] (c12) circle (.1cm);
\node at (4,1.5) {time $m$};
\vertex[fill,label=above:{\small R}](a01) at (4,0) {};
\vertex[fill,label=above:{\small C}](b11) at (3,.5) {};
\vertex[fill](b21) at (3,-.5) {};
\vertex[fill](c11) at (5,.5) {};
\vertex[fill](c21) at (5,-.5) {};
\draw (b11)--(a01) (b11)--(b21) (b21)--(a01);
\draw (c11)--(a01) (c11)--(c21) (c21)--(a01);
\draw[blue,thick] (b11) circle (.1cm);
\draw[red,thick] (a01) circle (.1cm);
\draw[thick,->] (1.5,1) -- (2.5,.2) node[anchor=north west] {};
\draw[thick,->] (1.5,-1) -- (2.5,-.2) node[anchor=north west] {};
\end{tikzpicture}
\caption{Position \ref{fset3} at even time $m$.}\label{even-conf3}
\end{figure}

\item[4.] At time $m$ the cop and robber are in Position \ref{fset4}. Figure \ref{even-conf4} gives the possible location where the robber was at time $m-1$ given that at time $m$ the cop and robber are in Position \ref{fset4}.  Hence, the probability that the robber remains free at time $m$ is given by the product of two probabilities: The first is the probability that the robber was free at time $m-1$ in Position \ref{fset4}, which is given by $\PP^4_{m-1}(F_k)$. The second probability is the probability that the robber did not move to the vertex occupied by the cop in his move at time $m-1$, which means that the robber moved to the other vertex in the same component he was occupying, so that the cop and robber are in Position \ref{fset4} at time $m$. Thus $\PP^4_{m}(F_k)=\left(\frac{1}{2}\right) \PP^4_{m-1}(F_k)$.

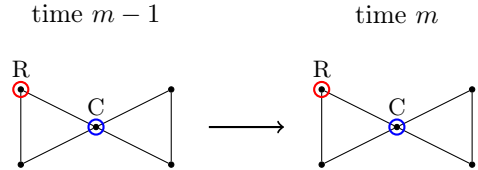
\begin{figure}[h]
\centering
\begin{tikzpicture}
\node at (0,1.5) {time $m-1$};
\vertex[fill,label=above:{\small C}](a0) at (0,0) {};
\vertex[fill,label=above:{\small R}](b1) at (-1,.5) {};
\vertex[fill](b2) at (-1,-.5) {};
\vertex[fill](c1) at (1,.5) {};
\vertex[fill](c2) at (1,-.5) {};
\draw (b1)--(a0) (b1)--(b2) (b2)--(a0);
\draw (c1)--(a0) (c1)--(c2) (c2)--(a0);
\draw[blue,thick] (a0) circle (.1cm);
\draw[red,thick] (b1) circle (.1cm);
\node at (4,1.5) {time $m$};
\vertex[fill,label=above:{\small C}](a01) at (4,0) {};
\vertex[fill,label=above:{\small R}](b11) at (3,.5) {};
\vertex[fill](b21) at (3,-.5) {};
\vertex[fill](c11) at (5,.5) {};
\vertex[fill](c21) at (5,-.5) {};
\draw (b11)--(a01) (b11)--(b21) (b21)--(a01);
\draw (c11)--(a01) (c11)--(c21) (c21)--(a01);
\draw[blue,thick] (a01) circle (.1cm);
\draw[red,thick] (b11) circle (.1cm);
\draw[thick,->] (1.5,0) -- (2.5,0) node[anchor=north west] {};
\end{tikzpicture}
\caption{Position \ref{fset4} at even time $m$.}\label{even-conf4}
\end{figure}
\qedhere
\end{enumerate}
\end{proof}

Using the recursive formulas of Theorem \ref{thm:FriendshipRecursive} we establish the following matrix formulas. 

\begin{theorem}\label{thm:M} Let $1\leq i\leq 4$. If $\theta$ denotes the proportion of the cop's movement that is random,
the the probability $P_m^i (F_k)$ is described by the following system of formulas in $k$.
If $m$ is odd, then 
\begin{align}
\begin{pmatrix}
\PP^1_m(F_k)\\
\PP^2_m(F_k)\\
\PP^3_m(F_k)\\
\PP^4_m(F_k)
\end{pmatrix}
&=\begin{pmatrix}
0&0&0&\left(\frac{1}{2k}\right)\theta\\
0&\left(\frac{1}{2}\right)\theta&0&\left(\frac{2k-2}{2k}\right)\theta\\
0&0&\left(\frac{1}{2}\right)\theta&0\\
\left(\frac{1}{2}\right)\theta&\left(1-\frac{1}{2}\theta\right)&0&0
\end{pmatrix}^m
\begin{pmatrix}
1\\
1\\
1\\
1
\end{pmatrix}.
\end{align}
If $m$ is even, then  
\begin{align}
\begin{pmatrix}
\PP^1_m(F_k)\\
\PP^2_m(F_k)\\
\PP^3_m(F_k)\\
\PP^4_m(F_k)
\end{pmatrix}
&=\begin{pmatrix}
0&0&\frac{1}{2k}&0\\
0&\frac{1}{2}&\frac{2k-2}{2k}&0\\
\frac{1}{2}&\frac{1}{4}&0&0\\
0&0&0&\frac{1}{2}
\end{pmatrix}^m
\begin{pmatrix}
1\\
1\\
1\\
1
\end{pmatrix}.
\end{align}

\end{theorem}

\begin{proof}
We see from Theorem \ref{thm:FriendshipRecursive} for odd $m$

\begin{align}
\begin{pmatrix}
\PP^1_m(F_k)\\
\PP^2_m(F_k)\\
\PP^3_m(F_k)\\
\PP^4_m(F_k)
\end{pmatrix}
&=\begin{pmatrix}
0&0&0&\left(\frac{1}{2k}\right)\theta\\
0&\left(\frac{1}{2}\right)\theta&0&\left(\frac{2k-2}{2k}\right)\theta\\
0&0&\left(\frac{1}{2}\right)\theta&0\\
\left(\frac{1}{2}\right)\theta&\left(1-\frac{1}{2}\theta\right)&0&0
\end{pmatrix}
\begin{pmatrix}
\PP^1_{m-1}(F_k)\\
\PP^2_{m-1}(F_k)\\
\PP^3_{m-1}(F_k)\\
\PP^4_{m-1}(F_k)
\end{pmatrix}
\end{align}
and for even $m$
\begin{align}
\begin{pmatrix}
\PP^1_m(F_k)\\
\PP^2_m(F_k)\\
\PP^3_m(F_k)\\
\PP^4_m(F_k)
\end{pmatrix}
&=\begin{pmatrix}
0&0&\frac{1}{2k}&0\\
0&\frac{1}{2}&\frac{2k-2}{2k}&0\\
\frac{1}{2}&\frac{1}{4}&0&0\\
0&0&0&\frac{1}{2}
\end{pmatrix}
\begin{pmatrix}
\PP^1_{m-1}(F_k)\\
\PP^2_{m-1}(F_k)\\
\PP^3_{m-1}(F_k)\\
\PP^4_{m-1}(F_k).
\end{pmatrix}
\end{align}
The result follows from induction on $m$ and the fact that $\PP^i_0(F_k)=1$ for all starting positions $1\leq i\leq 4$.
\end{proof}

Applying the Cayley-Hamilton Theorem to the matrices in Theorem \ref{thm:M} gives us the following simple recursive formula for each $\PP_m^i(F_k)$, for $1\leq i\leq 4$.

\begin{corollary} Let $F_k$ denote the friendship graph with $n$ components. Let $P_m^i (F_k)$ denote the probability that the robber is still free after $m$ time steps of the drunken cop and tipsy robber game starting in Position~$i$ for $1\leq i\leq 4$. If $\theta$ denotes the proportion of the cop's movement that is random, then each probability $P_m^i (F_k)$ satisfies the following recursive formulas.
For odd $m$
\begin{align}
\PP^i_{m+4}(F_k)&=
\theta \PP^i_{m+3}(F_k)
+
\left(\frac{(k-1)\theta}{k}-\frac{3\theta^2}{2}\right)\PP^i_{m+2}(F_k)
+
\left(\frac{(2-2k)\theta^2-(k-2)\theta^3}{4k}\right)\PP^i_{m+1}(F_k)
-
\frac{\theta^4}{16k}\PP^i_{m}(F_k)
\end{align}
and for even $m$
\begin{align}
\PP^i_{m+4}(F_k)=\PP^i_{m+3}(F_k)-\left(\frac{k+1}{8k}\right)\PP^i_{m+1}(F_k)+\frac{1}{16k}\PP^i_{m}(F_k).
\end{align}
\end{corollary}
\begin{proof} The Cayley-Hamilton Theorem states that any matrix $M$ must satisfy its own characteristic polynomial. The characteristic polynomials of 

\begin{align}
M_{odd}&=\begin{pmatrix}
0&0&0&\left(\frac{1}{2k}\right)\theta\\
0&\left(\frac{1}{2}\right)\theta&0&\left(\frac{2k-2}{2k}\right)\theta\\
0&0&\left(\frac{1}{2}\right)\theta&0\\
\left(\frac{1}{2}\right)\theta&\left(1-\frac{1}{2}\theta\right)&0&0
\end{pmatrix}
\end{align}
and
\begin{align}
M_{even}&=\begin{pmatrix}
0&0&\frac{1}{2k}&0\\
0&\frac{1}{2}&\frac{2k-2}{2k}&0\\
\frac{1}{2}&\frac{1}{4}&0&0\\
0&0&0&\frac{1}{2}
\end{pmatrix}\end{align}
are 
\begin{align}
p_{odd}(x)&=x^4
-\theta x^3
-
\left(\frac{(k-1)\theta}{k}-\frac{3\theta^2}{2}\right)x^2
-
\left(\frac{(2-2k)\theta^2-(k-2)\theta^3}{4k}\right)x
+
\frac{\theta^4}{16k}\\
p_{even}(x)&=x^4-x^3+\frac{k+1}{8k}x-\frac{1}{16k} \mbox{ respectively}.\qedhere
\end{align}
\end{proof}

\section{Open Problems}\label{sec:open}

One could consider analyzing the game of tipsy cop and drunk robber on other families of graphs, to include paths, grids, trees, and wheels. However, it may be of more interest to consider how the probabilities presented in this paper change if the cop and/or robber are sobering up as time steps occur. Hence, their decision making becomes less random as time goes on. In light of this we ask:

\begin{question}
If the cop and/or robber are sobering up with each move, how does the probability of the robber remaining free change from those presented in this manuscript? What about for other families of graphs?
\end{question}

Another direction of future research is motivated by the observation that on the star graph, the cop is guaranteed to catch the robber in two moves. Thus the star graph is best for the cop. This leads us to ask:

\begin{question}
Which graph(s) are best for the robber? Meaning, on what graph(s) can the robber maximize the number of moves until capture?
\end{question}

As noted earlier, one inspiration for the tipsy cop and drunken robber game is the biological scenario illustrated by this youtube video \url{ https://www.youtube.com/watch?v=Z_mXDvZQ6dU} where a neutrophil chases a bacteria cell moving in random directions.  While the bacteria in this video moves faster than the neutrophil, the neutrophil's movement is slightly more purposeful.  Inspired by this video, we ask the following questions.

\begin{question}
If a tipsy cop with tipsiness $\theta$ is pursuing a drunken robber on an an infinite grid $\ZZ \times \ZZ$, and the robber can take $m$ random steps in between each move by the tipsy cop, for what values of $\theta$ and $m$ does the cop win?  
\end{question}

\begin{question}
If a tipsy cop with tipsiness $\theta$ is pursuing a drunken robber on an an infinite grid $\ZZ \times \ZZ$ (or an isometric grid), and the robber can take $m$ random steps in between each move by the tipsy cop, for what values of $\theta$ and $m$ does the cop win? 
\end{question}

\begin{question}
In the biological system, multiple neutrophils might be `called' into the chase, and there might be multiple pathogens (bacteria cells). How would the game above change if multiple cops are introduced throughout the game?
\end{question}

\begin{question}
What would happen if the cop tipsiness $\theta$ depended on the distance between the cop and the robber? In the biological scenario this is a more realistic representation of immune cells sensing pathogens.
\end{question}

\begin{question}
In the current game, the cop and robber take turns moving and are not allowed to stay put. How would the game change if we allow both the cop and the robber to randomly choose at each time step whether they will move or stay? How would the game change if the probability of staying varies?
\end{question}


\begin{thebibliography}{9}

\bibitem{1} A. Berarducci and B. Intrigila, On the cop number of a graph. {\em Adv. Appl. Math.} {\bf 14}  (1993) 389--403.

\bibitem{2} A. Bonato, P. Golovach, G. Hahn, and J. Kratochvil, The capture time of a graph, {\em Discrete Math.} {\bf 309} (2009) 5588--5595.

\bibitem{3} G. Hahn, F. Laviolette, N. Sauer and R. E. Woodrow, On cop-win graphs, {\em Discrete Math.} {\bf 258}  (2002) 27--41.

\bibitem{4} G. Hahn and G. MacGillivray, A note on k-cop, l robber games on graphs, {\em Discrete Math.} {\bf 306}  (2006) 2492--2497.


\bibitem{AF}M. Aigner and M. Fromme, A game of cop's and robber's, \emph{Discrete Appl. Math.} 8 (1984), no. 1, 1-11.

\bibitem{BW} G. R. Brightwell and P. Winkler, Maximum hitting time for random walks on graphs,
{\em J. Random Structures and Algorithms} {\bf 1} \#3 (1990), 263--276.


\bibitem{CTW} D. Coppersmith, P. Tetali and P. Winkler, Collisions among random walks on a graph, {\em SIAM J. Disc. Math.} {\bf 6} \#3 (1993), 363--374.



\bibitem{Gal}I. Gal, \emph{Search Games}. Addison-Wesley, Reading, MA, 1982.




\bibitem{KW14}
{ N.~Komarov and P.~Winkler}, Catching the drunk robber on a graph,
  \emph{Electron. J. Combin.}, 21 (2014), pp.~Paper 3.30, 14.

\bibitem{NW83}R.~Nowakowski and P.~Winkler, Vertex-to-vertex pursuit in a graph, \emph{Discrete Math.} 13 (1983) 235-239.


\bibitem{Q86}A.~Quilliot, \emph{Some Results about Pursuit Games on Metric Spaces Obtained Through Graph Theory Techniques},
European Journal of Combinatorics
Volume 7, Issue 1, January 1986, Pages 55-66.

\end{thebibliography}
\end{document}